\documentclass[a4paper]{article}
\usepackage{amsmath,amsthm,amssymb,amscd,graphicx}

\numberwithin{equation}{section}

\newtheorem{thm}{Theorem}[section]
\newtheorem{lem}[thm]{Lemma}
\newtheorem{prop}[thm]{Proposition}
\newtheorem{cor}[thm]{Corollary}
\newtheorem{remark}[thm]{Remark}

\theoremstyle{definition}
\newtheorem{definition}[thm]{Definition}
\newtheorem{example}[thm]{Example}

\newenvironment{rem}{\begin{remark}\rm}{\end{remark}}

\def\address#1#2{\begingroup
\noindent\parbox[t]{12cm}{%
\small{\scshape\ignorespaces#1}\par\vskip1ex
\noindent\small{\itshape E-mail address}%
\/: #2\par\vskip4ex}\hfill%
\endgroup}%
\makeatother

\title{Deformations of homogeneous associative submanifolds in nearly parallel $G_{2}$-manifolds}
\author{Kotaro Kawai
\footnote{The author is supported by Grant-in-Aid for JSPS fellows (26-7067).}}
\date{}

\begin{document}

\maketitle

\begin{abstract}
A nearly parallel $G_{2}$-manifold $Y$ is 
a Riemannian 7-manifold 
whose cone $C(Y) = \mathbb{R}_{>0} \times Y$ has the holonomy group contained in ${\rm Spin(7)}$. 
In other words, it is a spin 7-manifold with a real Killing spinor. 

We have a special class of calibrated submanifolds 
called Cayley submanifolds in $C(Y)$. 
An associative submanifold in $Y$
is a minimal 3-submanifold whose cone is Cayley. 
We study its deformations, namely, Cayley cone deformations, 
explicitly 
when it is homogeneous in the 7-sphere $S^{7}$. 
\end{abstract}

\section{Introduction}
For any Riemannian manifold $(Y, g)$, consider 
its Riemannian cone $(C(Y), \overline{g}) = (\mathbb{R}_{>0} \times Y, dr^{2} + r^{2} g)$. 
A Riemannian 7-manifold $(Y, g)$ is 
called a nearly parallel $G_{2}$-manifold  
if the holonomy group of $\overline{g}$ is contained in ${\rm Spin(7)}$. 
The existence of a nearly parallel $G_{2}$-structure 
is equivalent to that of a spin structure with 
a real Killing spinor (\cite{Bar}), 
which is also used 
in supergravity and superstring theory in physics.

We have a canonical closed 4-form $\Phi$ on $C(Y)$, 
which defines a calibration. 
A 3-submanifold $M$ in $Y$ is called associative if 
its cone $C(M)$ is calibrated by $\Phi$. 
In other words, $C(M)$ is a Cayley submanifold in $C(Y)$. 
For example, 
special Legendrian submanifolds in Sasaki-Einstein manifolds are associative (Lemma \ref{sLeg_cpx_asso}), 
and Lagrangian submanifolds in the sine cones of nearly K\"{a}hler 6-manifolds 
are associative (Lemma \ref{submfd_sine cone NK}). 
Here, 
Lagrangian submanifolds are defined in terms of the vanishing of a non-closed 2-form 
which characterizes nearly K\"{a}hler geometry. 
These are also called totally real submanifolds.

The deformation of compact calibrated submanifolds was studied by Mclean \cite{Mclean}. 
Joyce \cite{Joyce1, Joyce2, Joyce3, Joyce4, Joyce5} 
introduced the notion of the stability index of a special Lagrangian cone 
to study deformations of a special Lagrangian submanifold with a conical singularity. 
Lotay \cite{Lotay_stab} generalized it to the coassociative case. 
Associative and Cayley submanifolds 
behave differently from 
special Lagrangian and coassociative submanifolds, 
and hence it is difficult to generalize it directly to the associative or Cayley case. 
Thus in this paper, 
we focus on the Cayley case and study the deformations of homogeneous Cayley cones explicitly. 
It may help to develop the general deformation theory of a Cayley submanifold with a conical singularity. 
Our approach is based on the representation theory. 
This is an analogue of Ohnita's approach to special Legendrian submanifolds in \cite{Ohnita_def}.

The homogeneous associative submanifolds in $S^{7}$ 
are classified by Lotay \cite{Lotay3}
into 8 types: 
$A_{1}, A_{2}$ and $A_{3}$ not lying in a totally geodesic nearly K\"{a}hler $S^{6}$, 
Lagrangian submanifolds $L_{1}, L_{2}, L_{3},$ and $L_{4}$ in $S^{6}$, 
and the totally geodesic $S^{3}$ 
(Proposition \ref{Lotay_classification_asso}).  
Infinitesimal Lagrangian deformations in $S^{6}$  
are studied in \cite{Lotay_stab}, 
and hence 
we study the infinitesimal deformations of the others and obtain the following.

\begin{thm} \label{rigid_summary}
As an associative submanifold, 
$A_{1}$ is rigid, while $A_{2}$ and $A_{3}$ are not rigid. 
The deformation space of $A_{2}$ is unobstructed, and 
all non-trivial associative deformations of $A_{2}$ are induced by 
the ${\rm PGL}(4, \mathbb{C})$-action on $\mathbb{C}P^{3}$ via the Hopf lift. 
\end{thm}

\begin{thm} \label{Lag_asso_rigid}
All the associative and non-Lagrangian 
deformations of 
the totally geodesic $S^{3}$, 
$L_{1}, L_{2}, L_{3},$ and $L_{4}$ are trivial. 
In other words, such deformations are induced from ${\rm Spin}(7) \setminus G_{2}$. 
\end{thm}

This paper is organized as follows. In Section 2, we review the fundamental
facts of $G_{2}$, ${\rm Spin}(7)$, Sasakian, and nearly K\"{a}hler geometry.

In Section 3, 
we characterize 
the space of all infinitesimal associative deformations 
as an eigenspace of a twisted Dirac operator $D$ (Proposition \ref{deform_asso_NP}). 
In Section 4 (5), 
we compute 
the difference of the dimension between infinitesimal associative
and special Legendrian (Lagrangian) deformations. 
These computations are useful to prove Theorem \ref{rigid_summary}  and \ref{Lag_asso_rigid} 
and give the geometrical meanings of some eigenspaces of some differential operators 
such as the Laplacian.

In Section 6, 
according to Lotay's classification, 
we calculate the dimensions of eigenspaces of homogeneous associative submanifolds 
by the representation theoretical method in Appendix B, 
and 
prove Theorem \ref{rigid_summary}  and \ref{Lag_asso_rigid}.

{\bf Notation}: 
Let $M$ be a manifold and $E$ be a vector bundle over $M$. 
We denote by $C(M, E)$ the space of all continuous sections of $E \rightarrow M$, 
and by $C^{\infty}(M, E)$ the space of all smooth sections of $E \rightarrow M$. 
Especially, we write $\mathfrak{X}(M) = C^{\infty}(M, TM)$. 

If a Lie group $G$ acts on $M$, 
we denote by $X^{*}$ the vector field generated by $X \in \mathfrak{g} = Lie(G)$.

\noindent{{\bf Acknowledgements}}: 
The author would like to thank the referee for 
the careful reading of an earlier version of this paper
and useful comments on it.

\section{Preliminaries}

\subsection{$G_{2}$ and ${\rm Spin}(7)$ geometry}

\begin{definition} \label{def on R7}
Define a $3$-form $\varphi_{0}$ on $\mathbb{R}^{7}$ by
\begin{eqnarray*}
\varphi_{0} = dx_{123} +dx_{1} (dx_{45} +dx_{67}) +dx_{2}(dx_{46} - dx_{57}) - dx_{3}(dx_{47} + dx_{56}), 
\end{eqnarray*}
where  $(x_{1}, \cdots, x_{7})$ is the standard coordinate system 
on $\mathbb{R}^{7}$ 
and wedge signs are omitted. 
The Hodge dual of $\varphi_{0}$ is given by 
\begin{eqnarray*}
*\varphi_{0} = dx_{4567} +dx_{23} (dx_{67} + dx_{45}) +dx_{13}(dx_{57} - dx_{46}) - dx_{12}(dx_{56} + dx_{47}). 
\end{eqnarray*}

Decompose $\mathbb{R}^{8} = \mathbb{R} \oplus \mathbb{R}^{7}$
and denote by $x_{0}$ the coordinate on $\mathbb{R}$. 
Define a self-dual $4$-form $\Phi_{0}$ on $\mathbb{R}^{8}$ by
\begin{align*}
\Phi_{0} = dx_{0} \wedge \varphi_{0} + * \varphi_{0}. 
\end{align*}
If we identify $\mathbb{R}^{8} \cong \mathbb{C}^{4}$ via 
$\mathbb{R}^{8} \ni (x_{0}, \cdots, x_{7}) 
\mapsto 
(x_{0} + i x_{1}, x_{2} + i x_{3}, x_{4} + i x_{5}, x_{6} + i x_{7}) 
=: (z_{1}, z_{2}, z_{3}, z_{4}) \in \mathbb{C}^{4}
$, 
then $\Phi_{0}$ is described as 
\begin{align*}
\Phi_{0} = \frac{1}{2} \omega_{0} \wedge \omega_{0} + {\rm Re} \Omega_{0}, 
\end{align*}
where $\omega_{0} = \frac{i}{2} \sum_{j = 1}^{4} dz_{j \overline{j}}$ and 
$\Omega_{0} = dz_{1234}$ are the standard 
K\"{a}hler form and the holomorphic volume form on $\mathbb{C}^{4}$, respectively. 
\end{definition}

The stabilizers of $\varphi_{0}$ and $\Phi_{0}$ are 
the exceptional Lie group $G_{2}$ and ${\rm Spin}(7)$, respectively: 
\begin{eqnarray*}
G_{2} = \{ g \in GL(7, \mathbb{R})  ;  g^{*}\varphi_{0} = \varphi_{0} \}, \qquad 
{\rm Spin}(7) = \{ g \in GL(8, \mathbb{R})  ;  g^{*}\Phi_{0} = \Phi_{0} \}. 
\end{eqnarray*}

The Lie group $G_{2}$ fixes 
the standard metric $g_{0} = \sum^7_{i=1} (dx_{i})^{2}$ and the orientation on $\mathbb{R}^{7}$. 
They are uniquely determined by $\varphi_{0}$ via 
\begin{eqnarray}\label{g varphi}
6 g_{0}(v_{1}, v_{2}) {\rm vol}_{g_{0}} = i(v_{1})\varphi_{0} \wedge i(v_{2})\varphi_{0} \wedge \varphi_{0}, 
\end{eqnarray}
where ${\rm vol}_{g_{0}}$ is a volume form of $g_{0}$, 
$i( \cdot )$ is the interior product, and $v_{i} \in T(\mathbb{R}^{7})$.

Similarly, 
${\rm Spin}(7)$ fixes 
the standard metric $h_{0} = \sum^7_{i=0} (dx_{i})^{2}$ and the orientation on $\mathbb{R}^{8}$. 
We have the following identities: 
\begin{eqnarray} \label{g Phi}
\Phi_{0} ^{2} = 14 {\rm vol}_{h_{0}}, \qquad
(i(w_{2}) i(w_{1}) \Phi_{0})^{2} \wedge \Phi_{0} = 6 
\|  w_{1} \wedge w_{2} \|_{h_{0}}^{2} 
{\rm vol}_{h_{0}}, 
\end{eqnarray}
where ${\rm vol}_{h_{0}}$ is a volume form of $h_{0}$ 
and $w_{i} \in T(\mathbb{R}^{8})$.

\begin{definition}
Let $Y$ be an oriented 7-manifold and 
$\varphi$ a 3-form on $Y$. 
A 3-form $\varphi$ is called a {\bf $G_{2}$-structure} on $Y$ if 
for each $y \in Y$, there exists an oriented isomorphism between $T_{y}Y$ and $\mathbb{R}^{7}$ 
identifying $\varphi_{y}$ with $\varphi_{0}$. 
From (\ref{g varphi}), $\varphi$ induces the metric $g$ 
and the volume form on $Y$.
A  $G_{2}$-structure $\varphi$ is said to be {\bf nearly parallel} 
if $d \varphi = 4 *  \varphi$. 
We call a manifold with a nearly parallel $G_{2}$-structure 
a {\bf nearly parallel $G_{2}$-manifold} for short. 
A  $G_{2}$-structure $\varphi$ is called {\bf torsion-free} if 
$ d\varphi = d*\varphi = 0$.

Let $X$ be an oriented 8-manifold and 
$\Phi$ a 4-form on $X$. 
A 4-form $\Phi$ is called a {\bf ${\rm Spin}(7)$-structure} on $X$ if 
for each $x \in X$, there exists an oriented isomorphism between $T_{x}X$ and $\mathbb{R}^{8}$ 
identifying $\Phi_{x}$ with $\Phi_{0}$. 
From (\ref{g Phi}), $\Phi$ induces the metric $h$ and 
the volume form on $X$. 
A ${\rm Spin}(7)$-structure $\Phi$ is called {\bf torsion-free} if 
$d\Phi = 0$. 
\end{definition}

\begin{lem} \cite{Salamon}
A  $G_{2}$-structure $\varphi$ is torsion-free if and 
only if {\rm Hol($g$)} $\subset G_{2}$.
A ${\rm Spin}(7)$-structure $\Phi$ is torsion-free if and 
only if {\rm Hol($h$)} $\subset {\rm Spin}(7)$.  
\end{lem}

\begin{lem} \cite{AleSem} \label{cha_NP}
The following are equivalent:
\begin{enumerate}
\item $d \varphi = 4 * \varphi$ (i.e. The 3-form $\varphi$ is a nearly parallel $G_{2}$-structure.), 
\item $\nabla \varphi = \frac{1}{4} d \varphi$, where $\nabla$ is the Levi-Civita connection of $g$, 
\item $\nabla \varphi = * \varphi$, 
\item $\nabla_{X} (* \varphi) = - g(X, \cdot) \wedge \varphi$ for any $X \in TY$, 
\item $i(X) \nabla_{X} \varphi = 0$ for any $X \in TY$, 
\item The Riemannian cone $C(Y) = \mathbb{R}_{>0} \times Y$ 
         admits a torsion-free {\rm Spin(7)}-structure                                                 
        $\Phi = r^{3} dr \wedge \varphi + r^{4} * \varphi$ with the induced cone metric 
        $\overline{g} = dr^{2} + r^{2} g$.      
\end{enumerate}
\end{lem}

%%%%%%%%%%%%%%%%%%%%submanifold%%%%%%%%%%%%%%%%%
Next, we give a summary of the facts 
about submanifolds. 
Let $Y$ be a manifold with a $G_{2}$-structure $\varphi$ and the induced metric $g$.

\begin{lem} \cite{Harvey Lawson} \label{def_asso_coasso}
For 
every oriented $k$-dimensional subspace $V^{k} \subset T_{p}Y$ 
where $p \in Y$ and  $k = 3, 4,$ 
we have
$
\varphi|_{V^{3}} \leq {\rm vol}_{V^{3}}, \ 
*\varphi|_{V^{4}} \leq {\rm vol}_{V^{4}}.
$
An oriented 3-submanifold $L^{3} \subset Y$ is called {\bf associative} 
if $\varphi|_{TL^{3}} = {\rm vol}_{L^{3}}$.
An oriented 4-submanifold $L^{4}$ is called {\bf coassociative} 
if $*\varphi|_{TL^{4}} = {\rm vol}_{L^{4}}$.
\end{lem}

\begin{lem} \cite{Harvey Lawson}  \label{def_chi}
Define a tangent bundle valued 3-form $\chi \in C^{\infty} (Y, \wedge^{3}T^{*}Y \otimes TY)$ by 
\begin{eqnarray*}
g(v_{1}, \chi(v_{2}, v_{3}, v_{4})) = *\varphi(v_{1}, v_{2}, v_{3}, v_{4})
\end{eqnarray*}
for $v_{i} \in TY$. 
If $L^{k} \subset Y$ is an oriented $k$-submanifold where $k = 3, 4$, 
then 
\begin{align*}
L^{3} \mbox{: associative} &\Leftrightarrow \chi|_{TL^{3}} = 0 \mbox{ and } \varphi|_{TL^{3}} > 0,\\ 
L^{4} \mbox{: coassociative} &\Leftrightarrow 
\varphi|_{TL^{4}} = 0 \mbox{ and } *\varphi|_{TL^{4}} > 0. 
\end{align*}
\end{lem}

\begin{definition} \label{def of cross product}
Define the cross product $\times : TY \times TY \rightarrow TY$ by 
\begin{align*}
g(u \times v, w) = \varphi (u, v, w)
\end{align*}
for $u, v, w \in TY$. 
This satisfies the following relation:
\begin{align} \label{chi_cross}
\chi (u, v, w) = u \times (v \times w) + g(u, v) w - g(u, w) v. 
\end{align}
\end{definition}

\begin{rem}
When $L^{3}$ is associative, 
there exists an orthonormal basis $\{ e_{1}, e_{2}, e_{3} \}$ 
satisfying $e_{3} = e_{1} \times e_{2}$ 
at any point in $L^{3}$.
\end{rem}

%%%%%%%%%%%%%%%%%%%%%%%

\begin{definition}
Let $X$ be a manifold with a ${\rm Spin}(7)$-structure $\Phi$.  
Then for 
every oriented $4$-dimensional subspace $W \subset T_{x}X$ 
where $x \in X,$ 
we have
$
\Phi|_{W} \leq {\rm vol}_{W}. 
$
An oriented 4-submanifold $N \subset X$ is called {\bf Cayley} 
if $\Phi|_{TN} = {\rm vol}_{N}$.
\end{definition}

\begin{lem} \label{equiv asso Cayley}
Let $(Y, \varphi, g)$ be a nearly parallel $G_{2}$-manifold  
and $L \subset Y$ be an oriented 3-submanifold. 
By Lemma \ref{cha_NP}, 
$C(Y)$ is a manifold with a torsion-free ${\rm Spin}(7)$-structure $\Phi$. 
Then $L \subset Y$ is associative if and only if 
$C(L) \subset C(Y)$ is Cayley. 
\end{lem}

\begin{lem} \cite{Lotay3}
There are no coassociative submanifolds of a nearly parallel $G_{2}$-manifold $(Y, \varphi, g)$. 
\end{lem}
\begin{proof}
If $L$ is a coassociative submanifold, we have $\varphi|_{TL} = 0$, 
which implies that $4 {\rm vol}_{L} = 4 * \varphi|_{TL} = d \varphi|_{TL} = 0$. This is a contradiction. 
\end{proof}

\subsection{Sasakian geometry}

\begin{definition}
An odd dimensional Riemannian manifold 
$(S,g)$ is a {\bf Sasakian manifold} if 
its Riemannian cone 
$(C(S), \overline{g}) = (\mathbb{R}_{>0} \times S, dr^{2} + r^{2} g)$ 
is a K\"ahler manifold with respect to some integrable complex structure $J$ over $C(S)$. 
\end{definition}
Here, $r$ is a standard coordinate of $\mathbb{R}_{>0}$ and we regard $r$ as
the function on $C(S)$. 
We identify $S$ with the submanifold $\{1\} \times S \subset C(S)$.

\begin{lem}
Let $(S,g)$ be a Sasakian $(2m + 1)$-manifold. 
If $g$ is Einstein, 
the cone $(C(S), \overline{g})$ is Ricci-flat. 
In addition, if 
there exists a holomorphic volume form $\Omega \in \Omega^{(m+1, 0)}(C(S))$ such that 
\begin{align} \label{CYcondition}
\omega^{m+1}/(m+1)! = (-1)^{m(m+1)/2} (i/2)^{m+1} \Omega \wedge \overline{\Omega}, 
\end{align}
where $\omega = \overline{g}(J \cdot, \cdot)$ is the associated K\"ahler form on $C(S)$, 
we call $(C(S), \overline{g}, J, \omega, \Omega)$ a Calabi-Yau manifold. 
\end{lem}

\begin{lem}[{\cite[Corollary 11.1.8]{Boyer Galicki}}]
If $S$ is a compact simply-connected Sasaki-Einstein manifold, 
$C(S)$ is a Calabi-Yau manifold. 
\end{lem}

\begin{rem}
The holomorphic volume form $\Omega$ is not unique. 
For any $\theta \in \mathbb{R}$, 
$e^{i \theta} \Omega$ also satisfies (\ref{CYcondition}). 
\end{rem}

Let $(S, g)$ be a Sasaki-Einstein 7-manifold 
with a Calabi-Yau structure on $C(S)$. 

\begin{lem} \label{NPstr_onSE}
There exists a 3-form $\varphi \in \Omega^{3}(S)$ such that  $(S, \varphi, g)$ is a nearly parallel
$G_{2}$-manifold. 
\end{lem}

\begin{proof}
Fix a holomorphic volume form $\Omega$. 
Then a 4-form  
\begin{align}\label{4-form on cone}
\Phi = \frac{1}{2} \omega \wedge \omega + {\rm Re} \Omega \in \Omega^{4}(C(S))
\end{align}
gives a torsion-free ${\rm Spin}(7)$-structure on $C(S)$. 
A 3-form $\varphi \in \Omega^{3}(S)$ defined by 
\begin{align*}
\Phi_{(r, p)} = r^{3} dr \wedge \varphi_{p} + r^{4} * \varphi_{p}, \qquad
\mbox{ where } (r, p) \in \mathbb{R}_{>0} \times S,
\end{align*}
gives the nearly parallel $G_{2}$-structure on $S$. 
\end{proof}

%%%%%%%%%%%%%%%%%%%%%%%submanifold%%%%%%%%%%%%%%%%%%%%%%

Next, we summarize the facts about submanifolds in Sasakian manifolds.

\begin{definition}
An $m$-submanifold $L \subset S$ is called {\bf Legendrian} if 
$C(L) \subset C(S)$ is Lagrangian: $\omega|_{TC(L)} = 0$. 
Fix a holomorphic volume form $\Omega$ on $C(S)$. 
An $m$-submanifold $L \subset S$ is called {\bf special Legendrian} if 
$C(L) \subset C(S)$ is special Lagrangian: ${\rm Re} \Omega|_{TC(L)} = {\rm vol}_{C(L)} 
\Leftrightarrow  
\omega|_{TC(L)} = 0, {\rm Im} \Omega|_{TC(L)} = 0$ and 
${\rm Re} \Omega|_{TC(L)} > 0$.
\end{definition}

The following is a well-known fact. 
For example, see \cite[Proposition 4.5]{Moriyama}.

\begin{lem}
Let $L \subset S$ be a Legendrian submanifold. 
Then 
$L$ is minimal if and only if ${\rm Im} (e^{i \theta} \Omega) = 0$ for some $\theta \in \mathbb{R}$. 
\end{lem}

By definition, we obtain the following result. 

\begin{lem} \label{sLeg_cpx_asso}
Let $L \subset S$ be an oriented 3-submanifold.  
If $L$ is special Legendrian 
or 
if the cone $C(L)$ is a complex submanifold in $C(S)$, 
$L$ is associative. 
\end{lem}

\begin{proof}
If $L$  is special Legendrian, we have
$\frac{1}{2} \omega \wedge \omega |_{TC(L)} = 0$ and 
${\rm Re} \Omega|_{TC(L)} = {\rm vol}_{C(L)}. $
If $C(L)$ is a complex submanifold, 
we have 
$\frac{1}{2} \omega \wedge \omega |_{TC(L)} = {\rm vol}_{C(L)}$ and
${\rm Re} \Omega|_{TC(L)} = 0$. 
By (\ref{4-form on cone}) and Lemma \ref{equiv asso Cayley}, we see that $L$ is associative 
in both cases. 
\end{proof}

\subsection{Infinitesimal deformation of special Legendrian submanifolds}

Let $(S, g)$ be a Sasaki-Einstein $(2m+1)$-manifold 
with a Calabi-Yau structure on $C(S)$. 
Fix a holomorphic volume form $\Omega$ 
and let $L \subset S$ be a special Legendrian submanifold. 

\begin{lem}[\cite{Ohnita_def}]  \label{deform_sLeg}
The vector space of all infinitesimal special Legendrian deformations of $L$ 
is identified with 
\begin{align} \label{sLeg_deform1}
\left \{ f \in C^{\infty}(L) ; \Delta_{+}f = (2m + 2) f \right \}, 
\end{align}
where $\Delta_{+}$ is the Hodge Laplacian for functions on $L$.
\end{lem}

We write the subscript $+$ of $\Delta_{+}$ 
since every eigenvalue of this Laplacian is non-negative 
if $L$ is compact. 

\begin{proof}
Let $\nu$ be the normal bundle of $L$ in $S$. 
Since $L$ is Legendrian, there is a canonical isomorphism 
$
\nu \ni v \mapsto 
\left( g(v, J(r \frac{\partial}{\partial r})|_{r = 1}), - g(Jv, \cdot) \right) \in \mathbb{R} \oplus T^{*}L. 
$
Via this identification, 
suppose that $V \in C^{\infty}(L, \nu)$ corresponds to 
$(f, \alpha) \in C^{\infty}(L) \oplus \Omega^{1} (L)$. 
Then we have 
\begin{align} 
0 = L_{V} \left. \left( i \left(r \frac{\partial}{\partial r} \right) \omega \right) \right|_{TL}
 &= - 2\alpha + df,  \label{sLeg_deform2} \\
0 = L_{V} \left. \left( i \left(r \frac{\partial}{\partial r} \right) {\rm Im} \Omega \right) \right|_{TL}
 &= d* \alpha + (m+1)f {\rm vol}_{L},
\end{align}
which implies the proof.
\end{proof}

The same result is obtained in \cite{FHY} 
by using the fact that a cone $C(L)$ of $L$ is 
special Lagrangian in $C(S)$ 
and applying the deformation theory of 
special Lagrangian submanifolds in \cite{Mclean}.

\subsection{Nearly K\"{a}hler geometry}

\begin{definition}

Let $(N, k, J,\sigma)$ be  a real 6-dimensional 
almost Hermitian manifold 
with a Hermitian metric $k$, an almost complex structure $J$ and 
an associated K\"{a}hler form $\sigma$.
Let $\psi^{\pm} \in \Omega^{3}(N)$ be 3-forms on $N$. 
A quintuple $(k, J, \sigma, \psi^{\pm})$ is called an 
 {\bf $SU(3)$-structure} if 
we have $\| \psi^{\pm} \| = 2$ and 
      $\Psi := \psi^{+} + \sqrt{-1} \psi^{-}$ is a $(3,0)$-form with respect to $J$.
\end{definition}

\begin{rem}
The $SU(3)$-structure with a K\"{a}hler structure 
and a holomorphic (3, 0)-form $\Psi$ is a Calabi-Yau structure.
In fact, we can prove 
\begin{align*}
\sigma \wedge \psi^{\pm} = 0, \ \ 
\sigma^{3}/3!={(-1)^{\frac{3(3-1)}{2}} (i/2 )^{3}} \Psi \wedge \bar{\Psi}. 
\end{align*}
\end{rem}

\begin{definition}
An $SU(3)$-structure satisfying $d \sigma = 3 \psi^{+}$ and  $d \psi^{-} = -2\sigma^{2}$ 
is called {\bf nearly K\"{a}hler}.
\end{definition}
\begin{lem}[\cite{But}]
Let $(N, k, J, \sigma)$ be a real 6-dimensional almost Hermitian manifold. 
It admits a nearly K\"{a}hler structure if and only if 
 $(\nabla_{X} J) X = 0$ for every vector field $X$ on $N$ 
      and $\nabla_{X} J \neq 0$ for every $0 \neq X \in TN$, 
      where $\nabla$ is the Levi-Civita connection of $k$. 
\end{lem}
\begin{lem}  \label{G2 str on NK cone}
Let $(N, k, J, \sigma, \psi^{\pm})$ be a nearly K\"{a}hler manifold. 
Then $C(N) = \mathbb{R}_{>0} \times N$ admits a torsion-free $G_{2}$-structure 
$(\varphi, \overline{k})$ with 
\begin{align*}
\overline{k}        &=  dr^{2} + r^{2}k, \\ 
\varphi             &= r^{2} dr \wedge \sigma + r^{3} \psi^{+} 
                         = \tfrac{1}{3}d( r^{3} \sigma), \\ 
* \varphi          &= r^{3} \psi^{-} \wedge dr + \tfrac{1}{2} r^{4} \sigma^{2}
                         = -\tfrac{1}{4} d ( r^{4} \psi^{-}). 
\end{align*}
\end{lem}
\begin{lem}[\cite{Boyer Galicki}] \label{G2 str on NK sine cone}
Let $(N, k, J, \sigma, \psi^{\pm})$ be a nearly K\"{a}hler manifold. 
Then $C_{s}(N) = (0, \pi) \times N$ (a sine cone of $N$) admits a nearly parallel $G_{2}$-structure $(\tilde{\varphi}, \tilde{k})$ with 
\begin{align*}
\tilde{k}        &=  dt^{2} + (\sin^{2}t) k, \\ 
\tilde{\varphi}  &= (\sin^{2}t)  dt \wedge \sigma + (\cos t \sin^{3} t)  \psi^{+} 
- (\sin^{4} t) \psi^{-}, \\ 
*\tilde{\varphi} &= \tfrac{1}{2} (\sin^{4} t) \sigma^{2} + (\sin^{3} t \cos t) \psi^{-} \wedge dt 
                  - (\sin^{4} t) dt \wedge \psi^{+}. 
\end{align*}
\end{lem}

We canonically identify $N$ with the submanifold $N \times \{ \frac{\pi}{2} \} \subset C_{s}(N)$. 

\begin{rem}
Since $C(N)$ admits a torsion-free $G_{2}$-structure, 
$\mathbb{R} \times C(N)$ admits a torsion-free Spin(7)-structure. 
The nearly parallel $G_{2}$-structure on $C_{s}(N)$ is induced 
via the identification 
$C(C_{s}(N)) = \mathbb{R}_{>0}  \times (0, \pi) \times N 
\ni (r, t, x) \mapsto (r \cos t, r \sin t, x)  
\in \mathbb{R} \times \mathbb{R}_{>0} \times N   = \mathbb{R} \times C(N)$. 
\end{rem}

\begin{lem} [\cite{MNS}] \label{cross_NK}
Let $(N, k, J, \sigma, \psi^{\pm})$ be a nearly K\"{a}hler manifold. 
Define a map $G : TN \times TN \rightarrow TN$ by 
$
k( G(u, v), w) = \psi^{+} (u, v, w)
$
for $u, v, w \in TN$. Then we have 
\begin{align*}
(\nabla_{X} J)(Y) = G(X, Y), \qquad
\nabla_{X} \psi^{+} = - k(X, \cdot) \wedge \sigma, 
\end{align*}
where $\nabla$ is the Levi-Civita connection of $k$ 
and $X, Y \in \mathfrak{X}(N)$. 
\end{lem}

\begin{lem} \label{submfd_cone NK}
Let $(N, k, J, \sigma, \psi^{\pm})$ be a nearly K\"{a}hler manifold. 
From Lemma \ref{G2 str on NK cone}, the cone $C(N) = \mathbb{R}_{>0} \times N$ 
admits a torsion-free $G_{2}$ structure. 
Let $\Sigma \subset N$ ($L \subset N$)
be an oriented 2(3)-submanifold. 
Then we have 
\begin{itemize}
\item $C(\Sigma) \subset C(N)$ 
      is associative if and only if 
      $\Sigma$ is a $J$-holomorphic curve.
\item $C(L^{3}) \subset C(N)$ 
      is a coassociative 4-fold if and only if 
      $L$ is Lagrangian: $\sigma|_{TL} = 0$. 
\end{itemize} 
\end{lem}

\begin{lem} \label{submfd_sine cone NK}
Let $(N, k, J, \sigma, \psi^{\pm})$ be a nearly K\"{a}hler manifold. 
From Lemma \ref{G2 str on NK sine cone}, the sine cone $C_{s}(N) = N \times (0, \pi)$ 
admits a nearly parallel $G_{2}$ structure. 
Let $\Sigma \subset N$ 
($L \subset N$)
be an oriented 2(3)-submanifold. 
Then it follows that 
\begin{itemize}
\item $C_{s}(\Sigma)  \subset C_{s}(N)$ 
      is associative if and only if 
      $\Sigma$ is a $J$-holomorphic curve.
\item $L \times \{ \frac{\pi}{2} \} \subset C_{s}(N)$ 
      is associative if and only if 
      $L$ is Lagrangian: $\sigma|_{TL} = 0$. 
\end{itemize} 
\end{lem}

\begin{rem}
On a nearly K\"{a}hler manifold, we know that $d \sigma = 3 \psi^{+}$, 
which implies that a Lagrangian submanifold $L$ satisfies $\psi^{+}|_{TL} = 0$.
Thus Lagrangian submanifolds in a nearly K\"{a}hler manifold are regarded as 
``special Lagrangian"(with phase $-i$). 
\end{rem}

We know the following as Lemma \ref{deform_sLeg}. 

\begin{lem} \label{deform_Lag}
The vector space of all infinitesimal Lagrangian deformations of $L$ 
in a nearly K\"{a}hler manifold 
is identified with 
\begin{align} 
\{ v \in \mathfrak{X}(L); {\rm rot}(v) = 3v \}, 
\end{align}
where ${\rm rot}(v) = \sum_{i = 1}^{3} e_{i} \times \nabla_{e_{i}}^{\top} v$, 
$\nabla^{\top}$ is the Levi-Civita connection of the metric $k_{L}$ on $L$ induced from $(N, k)$  
and $\{ e_{i} \}_{i = 1, 2, 3}$ is the local orthonormal frame of $TL$. 
\end{lem}

\begin{proof}
Since $L$ is Lagrangian, there is a canonical isomorphism 
between the tangent bundle $TL$ 
and the normal bundle of $L$ in $N$ via 
$v \mapsto J v$. 
Then a vector field $v \in \mathfrak{X}(L)$ on $L$ corresponds to 
an infinitesimal Lagrangian deformation of $L$ if and only if 
\begin{align*} 
0 = \left. L_{J v} \sigma \right|_{TL}
= 3 i(v)  {\rm vol}_{L} - d (k_{L}(v, \cdot)).
\end{align*}
Note that $\psi^{-}|_{TL} = -  {\rm vol}_{L}$. 
Then the equations 
$* (i(v) {\rm vol}_{L}) = k_{L}(v, \cdot)$ and 
$* d (k_{L}(v, \cdot)) = k_{L}({\rm rot}(v), \cdot)$ 
imply the proof.
\end{proof}

\section{Associative deformations in nearly parallel $G_{2}$-manifolds}

Let $(Y, \varphi, g)$ be a nearly parallel $G_{2}$-manifold, 
$\iota : M^{3} \hookrightarrow Y$ be an associative immersion, 
and $\{ \iota_{t} : M \hookrightarrow Y  \}_{t \in (- \epsilon, \epsilon )}$ be a smooth family of immersions 
with $\iota_{0} = \iota$.

\begin{definition}
A family $\{ \iota_{t} \}$ is called an {\bf associative deformation} of $\iota$ 
if $\iota_{t}$ is an associative immersion for each $t$. 
An associative deformation  $\{ \iota_{t} \}$ is called {\bf trivial} 
if  $\{ \iota_{t} \}$ is induced by a one-parameter family of 
automorphisms of $(Y, \varphi, g)$. 
If all infinitesimal associative deformations of $M$ 
come from 
trivial deformations, $M$ is called {\bf rigid}.
\end{definition}

First, we characterize the space of infinitesimal associative deformations of $M$.

\begin{prop} \label{deform_asso_NP}
Let $(Y, \varphi, g)$ be a nearly parallel $G_{2}$-manifold, and 
$M^{3} \subset Y$ be an associative submanifold. 
Denote by $\nu$  the normal bundle of $M$ in $Y$ 
and by $\nabla^{\perp}$ the connection on $\nu$ induced 
by the Levi-Civita connection $\nabla$ of $(Y, g)$.  

Taking any local orthonormal frame $\{ e_{1}, e_{2}, e_{3} \}$ of $TM$, 
define the operator 
$D : C^{\infty}(M, \nu) \rightarrow C^{\infty}(M, \nu)$ by 
\begin{align*}
D \psi := \sum_{i = 1}^{3} e_{i} \times \nabla^{\perp}_{e_{i}} \psi. 
\end{align*}
Then the vector space of all infinitesimal associative deformations 
of $M^{3} \hookrightarrow Y$ is identified with 
\begin{align*}
\{ \psi \in C^{\infty}(M, \nu) ; D \psi = -\psi \}. 
\end{align*}
\end{prop}

\begin{rem} \cite{Mclean}
There exists a rank 4 vector bundle $E \rightarrow M$ 
satisfying $\nu \cong \mathbb{S} \otimes_{\mathbb{H}} E$, 
where $\mathbb{S} \rightarrow M$ is a spinor bundle. 
Then $D$ is a twisted Dirac operator. 
\end{rem}

The proof of Proposition \ref{deform_asso_NP} 
comes from the following general theory of 
associative deformations. 

\begin{prop} [\cite{Gayet, Mclean}]
Let $(Y, \varphi, g)$ be a manifold with a $G_{2}$-structure and  
$M^{3} \subset Y$ be an associative submanifold. 
Then the vector space of all infinitesimal associative deformations 
of $M^{3} \hookrightarrow Y$ is identified with $\ker \tilde{D}$, 
where $\tilde{D} : C^{\infty}(M, \nu) \rightarrow C^{\infty}(M, \nu)$ is defined by 
\begin{align*}
\tilde{D} \psi := - \sum_{i = 1}^{3} e_{i} \times \nabla^{\perp}_{e_{i}} \psi 
                     + \sum_{k=1}^{4} (\nabla_{\psi} * \varphi) (\eta_{k}, \omega) \eta_{k}. 
\end{align*}
Here $\{ e_{1}, e_{2}, e_{3} \}$ is an oriented local orthonormal frame of $TM$, 
$\omega = e_{1} \wedge e_{2} \wedge e_{3}$, and 
$\{ \eta_{1}, \eta_{2}, \eta_{3}, \eta_{4} \}$ is 
a local orthonormal frame of $\nu$. 
\end{prop}

\begin{proof}
We give an outline of the proof. 
Define a map 
$F: C^{\infty}(M,\nu) \rightarrow C^{\infty}(M, TY|_{M}) $ as 
$F(\sigma) = \exp_{\sigma}^{*} \chi(\omega)$, 
where $\chi$ is defined in  Lemma \ref{def_chi}. 
We know that 
$\exp_{\sigma}(M)$ is associative 
if and only if $F(\sigma)$ vanishes. 
For any $\psi \in C^{\infty}(M, \nu)$, 
we may consider 
\begin{align*}
(dF)_{0} (\psi) = 0. 
\end{align*}
By a direct computation, 
the left hand side is equal to 
$
- \sum_{i = 1}^{3} e_{i} \times \nabla^{\perp}_{e_{i}} \psi 
+ \sum_{k=1}^{4} (\nabla_{\psi} * \varphi) (\eta_{k}, \omega) \eta_{k}, 
$ and hence the statement is proved. 
\end{proof}

By Lemma \ref{cha_NP}, we see the following lemma, 
which implies Proposition \ref{deform_asso_NP}. 

\begin{lem}
If $(Y, \varphi, g)$ is nearly parallel, then 
$\sum_{k=1}^{4} (\nabla_{\psi} * \varphi) (\eta_{k}, \omega) \eta_{k} = - \psi. $
\end{lem}

\begin{rem}
We can prove Proposition \ref{deform_asso_NP} 
by using the the fact 
that a cone $C(M)$ of $M$ is 
a Cayley submanifold 
in $C(Y)$ with a torsion-free ${\rm Spin(7)}$-structure. 
Applying the deformation theory of 
Cayley submanifolds in \cite{Mclean}, 
we consider the  Cayley cone deformation of $C(M)$. 
This is
an analogue of the proof of 
Lemma \ref{deform_sLeg} given by \cite{FHY}. 
\end{rem}

The operator $D$ has the following properties. 

\begin{lem}
The operator $D$ is elliptic. 
There exists a vector field $X \in \mathfrak{X}(M)$ on $M$ satisfying
\begin{align} \label{formal adjoint D}
g(D \psi, \psi') = - {\rm div}(X) + g(\psi, D \psi')
\end{align}
for any $\psi, \psi' \in C^{\infty}(M, \nu).$ 
In particular, when $M$ is compact, 
$D$ is self-adjoint. 
\end{lem}

\begin{proof}
The ellipticity of $D$ is shown in \cite{Gayet}. 
For any $\psi, \psi' \in C^{\infty}(M, \nu)$, we compute by 
Definition \ref{def of cross product} and Lemma \ref{diff_cross}
\begin{align*}
g(D \psi, \psi') &= g \left(\sum_{i=1}^{3} e_{i} \times \nabla_{e_{i}} \psi, \psi' \right)\\
&=
- \sum_{i=1}^{3} g(\nabla_{e_{i}} \psi, e_{i} \times \psi')\\
&=
\sum_{i=1}^{3} \left( - e_{i} (g(\psi, e_{i} \times \psi')) + g(\psi, \nabla_{e_{i}} e_{i} \times \psi') \right)
+
g(\psi, D \psi').
\end{align*}
Define the vector field $X \in \mathfrak{X}(M)$ on $M$ by 
$g(X,v)= g(\psi, v \times \psi')$ for $v \in TM$. Then we obtain (\ref{formal adjoint D}). 
\end{proof}

Since $D$ is a twisted Dirac operator, 
there is a close relation between $D^{2}$ and the Laplacian.  
Choose a local orthonormal frame $\{ e_{1}, e_{2}, e_{3} \}$ of $TM$ and 
define the operators 
$\nabla^{\perp *} \nabla^{\perp}, \mathcal{R}, \mathcal{A} : 
C^{\infty}(M, \nu) \rightarrow C^{\infty}(M, \nu)$ by 
\begin{align*}
\nabla^{\perp *} \nabla^{\perp} 
= \sum_{i = 1}^{3} (-\nabla_{e_{i}}^{\perp} \nabla_{e_{i}}^{\perp} + \nabla_{\nabla_{e_{i}}^{\top} e_{i}}^{\perp}), \qquad
\mathcal{R} = \pi_{\mathcal{V}} (\sum_{i = 1}^{3} R(e_{i}, \cdot) e_{i}),  
\mathcal{A} = {}^t\! A \circ A, 
\end{align*}
where 
$\nabla^{\perp}$ is the connection on the normal bundle $\nu$ induced 
by the Levi-Civita connection $\nabla$ of $(Y, g)$,  
$\nabla^{\top}$ is the orthogonal projection of $\nabla$ to $TM$, 
$R$ is the curvature tensor of $g$, 
$\pi_{\mathcal{V}}$ is the orthogonal projection to $\nu$, 
$A : \nu \ni \psi \mapsto (u \mapsto - \nabla_{u}^{\top} \psi) \in 
SM := \{ T : TM \rightarrow TM ; {}^t\! T = T \} $
(the second fundamental form), 
and 
${}^t\! A$ is the transpose of $A$.

\begin{prop} \label{Dirac_Lap}
Let $(Y, \varphi, g)$ be a nearly parallel $G_{2}$-manifold and 
$M^{3} \subset Y$ be an associative submanifold.  
Then we have 
\begin{align*}
(D - 3 id_{\nu}) (D + id_{\nu}) = \nabla^{\perp *} \nabla^{\perp} + \mathcal{R} - \mathcal{A}. 
\end{align*}
\end{prop}

The proof is given in the appendix. 
The right hand side $\mathcal{J}:= \nabla^{\perp *} \nabla^{\perp} + \mathcal{R} - \mathcal{A}$ 
is called a Jacobi operator, and $\ker \mathcal{J}$ is known to be 
the space of infinitesimal minimal deformations (\cite{Simons}). 
By this formula, 
$D \psi = -\psi$ implies $\mathcal{J} \psi = 0$, 
which ensures that associative deformations are minimal deformations.

\begin{rem}
When $M$ is compact, 
the space of all infinitesimal minimal and non-associative deformations of $M$ 
is identified with $\{ \psi \in C^{\infty}(M, \nu) ; D \psi = 3 \psi \}$. 
\end{rem}

\begin{proof}
Since $D$ is elliptic self-adjoint, 
there is an orthonormal basis $\{ \psi_{i} \}_{i=1}^{\infty} \subset C^{\infty}(M, \nu)$ 
of  $L^{2}(M, \nu)$ 
consisting of eigensections of $D$. 
The set of eigenvalues is discrete and 
the each eigenspace is finite dimensional.  
We may assume that $D \psi_{i} = \lambda_{i} \psi_{i}$ for 
$\lambda_{i} \in \mathbb{R}$. 
For any $\psi = \sum_{i=1}^{\infty} 
(\psi, \psi_{i})_{L^{2}} \psi_{i} \in C^{\infty}(M, \nu)$ where 
$(\cdot, \cdot)_{L^{2}}$ is the $L^{2}$ inner product, 
we have 
\begin{align*}
(D - 3 id_{\nu}) (D + id_{\nu}) \psi 
=& \sum_{i=1}^{\infty} ((D - 3 id_{\nu}) (D + id_{\nu}) \psi , \psi_{i})_{L_{2}} \psi_{i}\\
=& \sum_{i=1}^{\infty} (\psi , (D - 3 id_{\nu}) (D + id_{\nu}) \psi_{i})_{L_{2}} \psi_{i}\\
=& \sum_{i=1}^{\infty} (\lambda_{i}-3) (\lambda_{i}+1) (\psi, \psi_{i})_{L^{2}} \psi_{i}.
\end{align*} 
Since $\{ \psi_{i} \}_{i=1}^{\infty}$ is an orthonormal basis, 
we see that 
$(D - 3 id_{\nu}) (D + id_{\nu}) \psi = 0$ 
if and only if $(\lambda_{i}-3) (\lambda_{i}+1) (\psi, \psi_{i})_{L^{2}} = 0$ for each $i$.
Thus elements of $\ker (D - 3 id_{\nu}) (D + id_{\nu})$ are 
linear combinations of elements of $\ker (D - 3 id_{\nu})$ and $\ker(D + id_{\nu})$.
\end{proof}

\section{Associative deformations of special Legendrian submanifolds in 
Sasaki-Einstein manifolds}

Let $(S, g)$ be a Sasaki-Einstein 7-manifold with 
a Calabi-Yau structure $(\overline{g}, J, \omega, \Omega)$ on $C(S)$. 
Let $M \subset S$ be a special Legendrian submanifold. 
By Lemmas \ref{NPstr_onSE} and \ref{sLeg_cpx_asso}, 
$(S, \varphi, g)$ admits a nearly parallel $G_{2}$-structure 
for some $\varphi \in \Omega^{3}(S)$ and $M$ is associative. 
We study the infinitesimal associative deformations of $M$.

\subsection{Associative deformations of special Legendrians}

Let $\nu \rightarrow M$ be the normal bundle of $M$. 
First, we rewrite the operator $D : C^{\infty}(M, \nu) \rightarrow C^{\infty}(M, \nu)$
in Proposition \ref{deform_asso_NP}
in the special Legendrian case. 
Since $M$ is special Legendrian, there exists canonical isomorphism
$TM \oplus \mathbb{R} \ni (v, x) \mapsto Jv + x J(r \frac{\partial}{\partial r})|_{r = 1} \in \nu$. 
Via this identification, we obtain the following.

\begin{prop} \label{deform_asso_SE}
The corresponding operator $D : \mathfrak{X}(M) \oplus C^{\infty}(M) \rightarrow \mathfrak{X}(M) \oplus C^{\infty}(M)$ is described as 
\begin{align*}
D(v, f) = \left( - {\rm grad}(f) + {\rm rot}(v) + v, {\rm div}(v) + 3f \right), 
\end{align*}
where 
$g_{M}({\rm grad}(f),  \cdot) = df$, ${\rm div}(v) = {\rm tr} (\nabla^{\top} v)$, and 
${\rm rot}(v) = \sum_{i = 1}^{3} e_{i} \times \nabla_{e_{i}}^{\top} v$. 
Here, 
we denote by $g_{M}$ the metric on $M$ induced from $(Y, g)$, 
by $\nabla^{\top}$ the Levi-Civita connection of $g_{M}$, 
by $\{ e_{i} \}_{i = 1, 2, 3}$ the local orthonormal frame of $TM$, 
and 
$g_{M}(v \times w, \cdot) = \varphi(v, w, \cdot) = {\rm vol}_{M}(v, w, \cdot)$ $(v, w \in TM)$. 
\end{prop}

We first give all the statements in this section and then prove them.

\begin{cor} \label{diff_asso_sL}
We have  
\begin{align*}
  &\dim \{  \mbox{the infinitesimal associative deformations of } M \} \\
=&
\dim \{ f \in C^{\infty}(M); \Delta_{+}f = 8f \} 
+
\dim \{ v \in \mathfrak{X}(M); {\rm rot}(v) = -2v \}. 
\end{align*}
\end{cor}

\begin{rem}
From Lemma \ref{deform_sLeg}, 
$\dim \{ v \in \mathfrak{X}(M); {\rm rot}(v) = -2v \}$ 
gives the dimension of 
infinitesimal associative and non-special Legendrian deformations. 
\end{rem}

We have the same equations as in the vector analysis.

\begin{lem} \label{eq_vec_an}
For any $f \in C^{\infty}(M)$ and $v \in \mathfrak{X}(M)$, we have 
\begin{align*}
{\rm rot} ({\rm grad} (f)) &= 0, \qquad
{\rm div} ({\rm rot} (v))  = 0, \\
{\rm rot}({\rm rot}(v))    &= \nabla^{\top *} \nabla^{\top} v 
                                  + {\rm grad} ({\rm div}(v)) + \sum_{i = 1}^{3} R(v, e_{i}) e_{i}, 
\end{align*}
where 
$\{ e_{i} \}_{i = 1, 2, 3}$ is the local orthonormal frame of $TM$, 
$R$ is the curvature tensor, 
and 
$\nabla^{\top *} \nabla^{\top} 
= \sum_{i = 1}^{3} (-\nabla_{e_{i}}^{\top} \nabla_{e_{i}}^{\top} + \nabla_{\nabla_{e_{i}}^{\top} e_{i}}^{\top})$ 
is the rough Laplacian. 
\end{lem}

This lemma implies the following, which corresponds to Proposition \ref{Dirac_Lap}. 

\begin{cor} \label{D2_SE}
\begin{align*}
D^{2}(v, f) = \left ( -4 {\rm grad}(f) + v + {\rm rot}(v) + \nabla^{\top *} \nabla^{\top} v 
                    + \sum_{i = 1}^{3} R(v, e_{i}) e_{i}, \ 
                   \Delta_{+} f + 4 {\rm div}(v) + 9f \right). 
\end{align*}
\end{cor}

Now, we give proofs.

\begin{proof}[Proof of Proposition \ref{deform_asso_SE}]

Let $\{ e_{1}, e_{2}, e_{3}\} \subset TM$ be a local oriented orthonormal frame.  
Set $e_{4} := r \frac{\partial}{\partial r} |_{r = 1} $ and 
$\eta_{j} := J(e_{j})$ for $1 \leq j \leq 4$. 
Then $\{ \eta_{j} \}_{1 \leq j \leq 4}$ is a local oriented orthonormal frame of $\nu$. 
Let 
$\{e^{1}, \cdots, e^{4}, \eta^{1}, \cdots, \eta^{4} \}$ be the dual coframe, then 
we have 
\begin{align*}
\omega = \sum_{i = 1}^{4} e^{i} \wedge \eta^{i}, \qquad
\Omega = (e^{1} + i \eta^{1}) \wedge \cdots \wedge (e^{4} + i \eta^{4}). 
\end{align*}

Denoting 
$
\nabla_{e_{i}}^{\top} e_{j} = \sum_{k = 1}^{3} \Gamma_{i j}^{k} e_{k}
$
and 
$
\nabla_{e_{i}}^{\perp} \eta_{a} = \sum_{b = 1}^{4} \tilde{\Gamma}_{i a}^{b} \eta_{b}  
$
for 
$
1 \leq i, j \leq 3
$ 
and 
$1 \leq a \leq 4$, 
we see the following by a direct computation. 

\begin{lem}
We have 
\begin{align*}
(e_{i} \times \eta_{a}) 
= 
\left( 
\begin{array}{cccc}
\eta_{4}  & \eta_{3}   & -\eta_{2} & -\eta_{1} \\
-\eta_{3} & \eta_{4}  & \eta_{1}   & -\eta_{2} \\
\eta_{2}  & -\eta_{1} & \eta_{4}   & -\eta_{3} 
\end{array} 
\right), \\
\tilde{\Gamma}_{i j}^{k} = \Gamma_{i j}^{k}, \qquad
\tilde{\Gamma}_{i j}^{4} = - \delta_{i j}, \qquad
\tilde{\Gamma}_{i 4}^{k} = \delta_{i k}, \qquad
\tilde{\Gamma}_{i 4}^{4} = 0, 
\end{align*}
for $1 \leq i, j, k \leq 3, 1 \leq a \leq 4$.
\end{lem}
Then via the identification 
$\mathfrak{X}(M) \oplus C^{\infty}(M) 
\ni 
(\sum_{j = 1}^{3} v_{j} e_{j}, f) 
\mapsto
\sum_{j = 1}^{3} v_{j} \eta_{j} + f \eta_{4}
\in C^{\infty}(M, \nu)$
where $v_{j}, f \in C^{\infty}(M)$, 
we have 
\begin{align*}
&D \left(\sum_{j = 1}^{3} v_{j} \eta_{j} + f \eta_{4} \right)\\
=&
\sum_{i = 1}^{3} e_{i} \times \nabla_{e_{i}}^{\perp} 
\left(\sum_{j = 1}^{3} v_{j} \eta_{j} + f \eta_{4} \right) \\
=&
\sum_{i, j = 1}^{3} e_{i}(v_{j}) e_{i} \times \eta_{j} 
+
\sum_{i, j = 1}^{3} v_{j} e_{i} \times 
\left(\sum_{k = 1}^{3} \Gamma_{i j}^{k} \eta_{k} - \delta_{i j} \eta_{4} \right) 
+
\sum_{i = 1}^{3} e_{i} (f) e_{i} \times \eta_{4}
+ 
\sum_{i = 1}^{3} f e_{i} \times \eta_{i}\\
=&
\left\{ e_{2}(v_{3}) - e_{3}(v_{2}) + \sum_{j = 1}^{3} v_{j} (\Gamma_{2 j}^{3} - \Gamma_{3 j}^{2}) \right\} \eta_{1} 
+
\left\{ e_{3}(v_{1}) - e_{1}(v_{3}) + \sum_{j = 1}^{3} v_{j} (\Gamma_{3 j}^{1} - \Gamma_{1 j}^{3}) \right\} \eta_{2}\\ 
&+
\left\{ e_{1}(v_{2}) - e_{2}(v_{1}) + \sum_{j = 1}^{3} v_{j} (\Gamma_{1 j}^{2} - \Gamma_{2 j}^{1}) \right\} \eta_{3} 
+
\left\{
\sum_{i = 1}^{3} e_{i}(v_{i}) + \sum_{i, j = 1}^{3} v_{j} \Gamma_{i j}^{i} \right\} \eta_{4} \\
&+
\sum_{i = 1}^{3} v_{i} \eta_{i} 
- 
\sum_{i = 1}^{3} e_{i}(f) \eta_{i} 
+
3 f \eta_{4}, 
\end{align*}
which gives the proof. 
\end{proof}

\begin{proof}[Proof of Lemma \ref{eq_vec_an}]
The first two equations are easy to prove. 
We only prove the third equation. 
By Lemma \ref{diff_cross} 
and the fact that $M$ is associative, 
it follows that 
\begin{align*}
{\rm rot}({\rm rot}(v)) 
&=
\sum_{i, j = 1}^{3} e_{i} \times 
(\nabla_{e_{i}}^{\top} e_{j} \times \nabla_{e_{j}}^{\top} v + e_{j} \times \nabla_{e_{i}}^{\top} \nabla_{e_{j}}^{\top} v). 
\end{align*}
From (\ref{chi_cross}) , we have 
$u \times (v \times w) = -g(u, v)w + g(u, w)v$ for $u, v, w \in TM$
on an associative submanifold $M$. Hence we have 
\begin{align*}
&{\rm rot}({\rm rot}(v)) \\
=& 
\sum_{i, j = 1}^{3} 
\left (
\Gamma_{i i}^{j} \nabla_{e_{j}}^{\top} v + g(e_{i}, \nabla_{e_{j}}^{\top} v) \nabla_{e_{i}}^{\top} e_{j} 
-
\delta_{i j} \nabla_{e_{i}}^{\top} \nabla_{e_{j}}^{\top} v 
+
g(e_{i}, \nabla_{e_{i}}^{\top} \nabla_{e_{j}}^{\top} v) e_{j} 
\right )\\
=&
\nabla^{\top *} \nabla^{\top} v 
+
\sum_{i, j = 1}^{3} 
g((\nabla_{e_{i}}^{\top} \nabla_{e_{j}}^{\top} - \nabla_{\nabla_{e_{i}}^{\top} e_{j}}^{\top})v, e_{i}) e_{j}, 
\end{align*}
where we use the fact that $\Gamma_{i i}^{j} = - \Gamma_{i j}^{i}$ 
since 
$\{ e_{i} \}_{i = 1, 2, 3}$ is the local orthonormal frame of $TM$. 
On the other hand, we have 
\begin{align*}
{\rm grad} ({\rm div}(v))
=
\sum_{i, j = 1}^{3}
e_{i} (g(\nabla_{e_{j}}^{\top}v, e_{j})) e_{i} 
=
\sum_{i, j = 1}^{3} 
g((\nabla_{e_{i}}^{\top} \nabla_{e_{j}}^{\top} - \nabla_{\nabla_{e_{i}}^{\top} e_{j}}^{\top})v, e_{j}) e_{i}, 
\end{align*}
which implies the proof. 
\end{proof}

The proof of Corollary \ref{D2_SE} is straightforward and we omit it.

\begin{proof}[Proof of Corollary \ref{diff_asso_sL}]

By Proposition \ref{deform_asso_SE}, 
$D(v, f) = -(v, f)$ is equivalent to 
$ 
{\rm rot}(v) + 2v = {\rm grad}(f), 
{\rm div}(v)        = -4f. 
$
Considering the divergence of the first equation, we have
$\Delta_{+}f = -2 {\rm div}(v)$, 
which implies that 
$D(v, f) = -(v, f)$ 
is equivalent to 
\begin{align} \label{-1esp_sL}
\left\{
\begin{aligned}
{\rm rot}(v) + 2v &= {\rm grad}(f), \\
\Delta_{+} f &= 8f.
\end{aligned}
\right.
\end{align}
The second equation is given in (\ref{sLeg_deform1}). 
For any $f \in C^{\infty}(M)$ satisfying $\Delta_{+}f = 8f$, 
$(v, f) = (\frac{1}{2} {\rm grad}(f), f)$ is the solution of (\ref{-1esp_sL}), 
which corresponds to the infinitesimal special Legendrian deformations of $M$.  (See (\ref{sLeg_deform2})). 
\end{proof}

\subsection{Associative deformation of homogeneous special Legendrians}

The method and the notation in this subsection are summarized in the appendix. 
We give an explicit description of the operator ${\rm rot}$ 
when 
$M $ is the reductive homogeneous space $G/K$, 
where $G \subset {\rm Aut}(S, \varphi, g)$ and 
$K \subset G$ is a closed subgroup.  
Take an ${\rm Ad}(K)$-invariant vector subspace of $\mathfrak{p} \subset \mathfrak{g}$ 
satisfying $\mathfrak{g} = \mathfrak{k} \oplus \mathfrak{p}$.

It is well-known that 
there is an one-to-one correspondence 
between 
${\rm Ad}(K)$-invariant inner products on $\mathfrak{p}$ and 
$G$-invariant metrics on $M = G/K$. 
Since $G \subset {\rm Aut}(S, \varphi, g)$, 
there exists a $G$-invariant metric $g_{M}$ on $M$ 
induced from $(S, g)$. 
Denote by $\langle \cdot, \cdot \rangle$ the corresponding 
${\rm Ad}(K)$-invariant inner product 
and 
by $\{ e_{1}, e_{2}, e_{3} \} \subset \mathfrak{p}$ 
an oriented orthonormal basis of $\mathfrak{p}$. 
Then we have the following.

\begin{lem}
The map 
$
G \times_{{\rm Ad}} \mathfrak{p} 
\ni [g, X] \mapsto 
\frac{d}{dt} g \cdot \exp(tX) K |_{t = 0} 
\in TM$
is an isomorphism. 
Thus  
the tangent bundle $TM$ of $M$ is a homogeneous vector bundle. 
\end{lem}

\begin{prop} \label{rot_homog_prop}
The operator ${\rm rot} : \mathfrak{X}(M) \rightarrow \mathfrak{X}(M)$ 
is a homogeneous differential operator 
and induces the map 
$\widetilde{{\rm rot}} : C^{\infty}(G, \mathfrak{p})^{(K, {\rm Ad})} 
\rightarrow C^{\infty}(G, \mathfrak{p})^{(K, {\rm Ad})}$. 
If we define 
$\overline{{\rm rot}} \in ({\rm End}(\mathfrak{p}) \otimes U(\mathfrak{g}))^{K}$ by 
\begin{align*} 
\overline{{\rm rot}} = 
\sum_{i \in \mathbb{Z}/3} e_{i}^{*} \otimes (e_{i+1} \wedge e_{i+2})
-
\sum_{i \in \mathbb{Z}/3}
\langle [e_{i+1}, e_{i+2}]_{\mathfrak{p}}, \cdot \rangle e_{i} \otimes 1, 
\end{align*}
where $\{ e_{i}^{*} \}_{i = 1, 2, 3}$ is the dual basis of $\{ e_{i} \}_{i = 1, 2, 3}$, 
we have 
\begin{align}\label{overline_rot_homog tilde_rot}
\overline{{\rm rot}} |_{C^{\infty}(G, \mathfrak{p})^{(K, {\rm Ad})}} = \widetilde{{\rm rot}}.
\end{align} 
\end{prop}

\begin{rem} \label{rot_homog}
Set $[e_{i}, e_{j}]_{\mathfrak{p}} = \sum_{k = 1}^{3} c_{i j}^{k} e_{k}$. 
Then with respect to $\{ e_{1}, e_{2}, e_{3}\}$, 
$\overline{{\rm rot}}$ is described as the following 
$U(\mathfrak{g})$-valued matrix:
\begin{align*}
\overline{{\rm rot}} = 
\left(
\begin{array}{ccc}
      0& -e_{3} & e_{2}  \\
  e_{3}&      0  & -e_{1}\\
-e_{2}&   e_{1} & 0     \\
\end{array} 
\right)
-
\left(
\begin{array}{ccc}
c_{2 3}^{1}& c_{2 3}^{2}& c_{2 3}^{3}\\
c_{3 1}^{1}& c_{3 1}^{2}& c_{3 1}^{3}\\
c_{1 2}^{1}& c_{1 2}^{2}& c_{1 2}^{3}\\
\end{array} 
\right).
\end{align*}
\end{rem}

\begin{proof}[Proof of Proposition \ref{rot_homog_prop}]

It is straightforward to show that ${\rm rot}$ is a homogeneous differential operator. 
Since $\overline{{\rm rot}}$ is independent of the choice of $\{ e_{i} \}_{i = 1, 2, 3}$ 
and ${\rm Ad}(K)$ preserves the orientation and the metric, 
we see that  $\overline{{\rm rot}}$ is $K$-invariant.

From Remark \ref{1pt_homog_diff}, 
a homogeneous differential operator 
is completely determined 
by its value at a point, 
and so we only have to compute 
${\rm rot}$ at $eK \in G/K=M$. 

For any $\tilde{v} = \sum_{i = 1}^{3} v_{i} e_{i} \in C^{\infty}(G, \mathfrak{p})^{(K, {\rm Ad})}$
where $v_{i} \in C^{\infty}(G)$, 
denote by $v \in \mathfrak{X}(M)$ the induced vector field: 
\begin{align*}
v(gK) = \left. \frac{d}{dt} g \cdot \exp \left( t \sum_{i=1}^{3} v_{i}(g) e_{i} \right) \cdot K \right|_{t = 0}. 
%
%= \sum_{i=1}^{3} v_{i}(g) e_{i}, 
\end{align*}

Take local coordinates $(y_{1}, y_{2}, y_{3})$ around $eK$ defined by 
$(y_{1}, y_{2}, y_{3}) \mapsto \exp \left( \sum_{i=1}^{3} y_{i} e_{i} \right)$. 
Let $\pi : G \rightarrow G/K = M$ be the projection and 
$\tau_{g} : M \rightarrow M$ for  $g \in G$ be the left translation.
Denoting $\nabla_{\frac{\partial}{\partial y_{i}}}^{\top} \frac{\partial}{\partial y_{j}} 
= \sum_{k = 1}^{3} \Gamma_{i j}^{k} \frac{\partial}{\partial y_{k}}, 
$
we see the following. 

\begin{lem} [\cite{Helgason, Muto_Urakawa}]
For a sufficiently small $X \in \mathfrak{p}$, we have
\begin{align*}
\left( \frac{\partial}{\partial y_{i}} \right)_{\exp(X) \cdot K} 
=&
((\tau_{\exp(X)})_{*})_{eK} (\pi_{*})_{e} \left( \sum_{m = 0}^{\infty} \frac{(-{\rm ad}(X))^{m}}{(m+1)!} e_{i} \right)\\
=&
((\tau_{\exp(X)})_{*})_{eK} \left( \sum_{m = 0}^{\infty} \frac{(-{\rm ad}(X))^{m}}{(m+1)!} 
\left( \frac{\partial}{\partial y_{i}} \right)_{e K}  \right),
\\
g_{M} \left(\frac{\partial}{\partial y_{i}}, \ \frac{\partial}{\partial y_{j}} \right)_{eK} =& \delta_{i j}, 
\qquad
\Gamma_{i j}^{k}(eK) = \frac{1}{2}(c_{k i}^{j} + c_{k j}^{i}). 
\end{align*}
\end{lem}

Now we compute ${\rm rot} (v)$ at $eK \in G/K =M$. 
First, we compute $(\nabla_{\frac{\partial}{\partial y_{i}}}^{\top} v)_{eK}$. 
Since the metric $g_{M}$ is $G$-invariant, we have 
\begin{align*}
\langle e_{i}, e_{j} \rangle
=
g_{M} \left( \frac{\partial}{\partial y_{i}}, \frac{\partial}{\partial y_{j}} \right)_{eK}
=
g_{M} \left(((\tau_{\exp(X)})_{*})_{eK} \left(\frac{\partial}{\partial y_{i}} \right)_{e K}, 
((\tau_{\exp(X)})_{*})_{eK} \left(\frac{\partial}{\partial y_{j}} \right)_{e K} \right) 
\end{align*}
for any $X \in \mathfrak{p}$. 
Then for sufficiently small $t \in \mathbb{R}$, 
it follows that  
\begin{align*}
g_{M} \left( v, \frac{\partial}{\partial y_{j}} \right )_{exp(t e_{i}) \cdot K}
&=
g_{M} \left(\sum_{k = 1}^{3} v_{k}(\exp (t e_{i})) (\tau_{\exp(t e_{i})})_{*} 
\left( \frac{\partial}{\partial y_{k}} \right)_{eK}, 
\left( \frac{\partial}{\partial y_{j}}  \right)_{\exp(t e_{i}) \cdot K} \right) \\
&=
\sum_{k = 1}^{3} v_{k}(\exp (t e_{i})) 
\left \langle
e_{k}, \sum_{m = 0}^{\infty} \frac{(-t \cdot {\rm ad}(e_{i}))^{m}}{(m+1)!} e_{j}
\right \rangle, \\
g_{M} \left(\nabla_{\frac{\partial}{\partial y_{i}}}^{\top} v, \frac{\partial}{\partial y_{j}} \right)_{eK}
&= 
\left( \frac{\partial}{\partial y_{i}} \right)_{eK} g_{M} \left(v, \frac{\partial}{\partial y_{j}} \right) 
- 
g_{M} \left(v, \nabla_{\frac{\partial}{\partial y_{i}}}^{\top} \frac{\partial}{\partial y_{j}} \right)_{eK}\\ 
&= 
\sum_{k = 1}^{3} \left( e_{i}(v_{k}) \delta_{k j} + v_{k} \left \langle e_{k}, - \frac{1}{2} [e_{i}, e_{j}] \right \rangle \right)
-g_{M}\left(v, \sum_{k = 1}^{3} \Gamma_{i j}^{k}(eK) \frac{\partial}{\partial y_{k}} \right)\\
&=
e_{i}(v_{j}) -\frac{1}{2} \sum_{k = 1}^{3} v_{k}(c_{i j}^{k} + c_{k i}^{j} + c_{k j}^{i}). 
\end{align*}
Hence we obtain 
\begin{align*}
(\nabla_{\frac{\partial}{\partial y_{i}}}^{\top} v)_{eK}
=
\sum_{j = 1}^{3}e_{i}(v_{j}) \frac{\partial}{\partial y_{j}}
- 
\frac{1}{2} \sum_{j, k = 1}^{3} v_{k}(c_{i j}^{k} + c_{k i}^{j} + c_{k j}^{i}) \frac{\partial}{\partial y_{j}}. 
\end{align*}
Thus 
we compute 
\begin{align*}
({\rm rot}(v))_{eK}
&=
\left( \frac{\partial}{\partial y_{i}} \right)_{eK} \times 
(\nabla_{\frac{\partial}{\partial y_{i}}}^{\top} v)_{eK}\\
&=
\left(
- e_{3}(v_{2}) + e_{2}(v_{3}) - \sum_{j = 1}^{3} c_{2 3}^{j} 
\right )  \frac{\partial}{\partial y_{1}}
+
\left(
e_{3}(v_{1}) - e_{1}(v_{3}) - \sum_{j = 1}^{3} c_{3 1}^{j} 
\right )  \frac{\partial}{\partial y_{2}} \\
&+
\left(
- e_{2}(v_{1}) + e_{1}(v_{2}) - \sum_{j = 1}^{3} c_{1 2}^{j} 
\right )  \frac{\partial}{\partial y_{3}}, 
\end{align*}
which implies the proof. 
\end{proof}

\section{Associative deformations in the sine cone of nearly K\"{a}hler manifolds}

Let $(N, k, J, \sigma, \psi^{\pm})$ be a nearly K\"{a}hler manifold and 
$L \subset N$ be a Lagrangian submanifold. 
From Lemma \ref{G2 str on NK sine cone} and \ref{submfd_sine cone NK}, 
the sine cone $C_{s}(N) = (0, \pi) \times N$ admits nearly parallel $G_{2}$-structure 
$(\tilde{\varphi}, \tilde{k})$ 
and 
$\{ \frac{\pi}{2} \} \times L \subset C_{s}(N)$ is associative. 
We study the infinitesimal associative deformations of $\{ \frac{\pi}{2} \} \times L$.

Let $\nu \rightarrow \{ \frac{\pi}{2} \} \times L$ be the normal bundle of 
$\{ \frac{\pi}{2} \} \times L \subset C_{s}(N)$. 
First, we rewrite the operator 
$D : C^{\infty}(\{ \frac{\pi}{2} \} \times L, \nu) \rightarrow C^{\infty}(\{ \frac{\pi}{2} \} \times L, \nu)$
in Proposition \ref{deform_asso_NP}
in this case. 
Since $L$ is Lagrangian, there exists canonical isomorphism
$TL \oplus \mathbb{R} \ni (v, x) \mapsto Jv + x \frac{\partial}{\partial t} |_{t = \frac{\pi}{2}} \in \nu$. 
Via this identification, we obtain the following.

\begin{prop} \label{deform_asso_sine cone NK}
The corresponding operator $D : \mathfrak{X}(L) \oplus C^{\infty}(L) \rightarrow \mathfrak{X}(L) \oplus C^{\infty}(L)$ is described as 
\begin{align*}
D(v, f) = \left( - {\rm grad}(f) - {\rm rot}(v) + 2v, {\rm div}(v) \right), 
\end{align*}
where 
we use the notation in Proposition \ref{deform_asso_SE}. 
\end{prop}

By Proposition \ref{deform_asso_sine cone NK}, 
we prove Corollary \ref{diff_asso_Lag} 
as in the case of Corollary \ref{diff_asso_sL}.

%%%%%%%%%%%%%%%%%%%%%%%%%%%%%%%%%%%%%%%%%%%%%%

\begin{cor} \label{diff_asso_Lag}
We have  
\begin{align*}
  &\dim \{  \mbox{the infinitesimal associative deformations of } \{ \tfrac{\pi}{2} \} \times L \} \\
=&
\dim \{ f \in C^{\infty}(L); \Delta_{+}f = 3f \} 
+
\dim \{ v \in \mathfrak{X}(L); {\rm rot}(v) = 3v \}. 
\end{align*}
\end{cor}

\begin{rem} \label{dim asso_nonLag}
From Lemma \ref{deform_Lag}, 
$\dim \{ f \in C^{\infty}(L); \Delta_{+}f = 3f \}$ gives 
the dimension of infinitesimal associative  
and non-Lagrangian deformations. 
\end{rem}

\begin{proof}[Proof of Proposition \ref{deform_asso_sine cone NK}]
Let $\{ e_{1}, e_{2}, e_{3}\} \subset TL$ be a local oriented orthonormal frame 
such that $\underline{\psi^{-}(e_{1}, e_{2}, e_{3}) = -1}$. 
Set $\eta_{j} := J(e_{j})$ for $1 \leq j \leq 3$ and 
$\eta_{4} := \frac{\partial}{\partial t} |_{t = \frac{\pi}{2}}$. 
Then $\{ \eta_{j} \}_{1 \leq j \leq 4}$ is a local oriented orthonormal frame of $\nu$. 
Let 
$\{e^{1}, \cdots, e^{3}, \eta^{1}, \cdots, \eta^{4} \}$ be the dual coframe, then 
we have 
\begin{align*}
\sigma = \sum_{i = 1}^{3} e^{i} \wedge \eta^{i}, \qquad
\Psi = \psi^{+} + i \psi^{-} = -i (e^{1} + i \eta^{1}) \wedge (e^{2} + i \eta^{2})  \wedge (e^{3} + i \eta^{3}). 
\end{align*}
Hence at a point of $L \times \{ \frac{\pi}{2} \}$, we have 
\begin{align*} 
\tilde{\varphi} = \eta^{4} \wedge \sum_{i = 1}^{3} e^{i} \wedge \eta^{i} 
+ e^{1} (e^{23} - \eta^{23}) - \eta^{1} (e^{2} \wedge \eta^{3} + \eta^{2} \wedge e^{3}). 
\end{align*}

As in the Sasakian case, 
the definition of the Levi-Civita connection gives the following.

\begin{lem}
For any $X, Y \in \mathfrak{X}(N \times \{ \frac{\pi}{2} \})$, we have 
\begin{align*}
\nabla^{C_{s}(N)}_{X} Y |_{N \times \{ \frac{\pi}{2} \}} = \nabla^{N}_{X} Y, \qquad
\nabla^{C_{s}(N)}_{X} \frac{\partial}{\partial t} |_{N \times \{ \frac{\pi}{2} \}} = 0, \qquad
\nabla^{C_{s}(N)}_{\frac{\partial}{\partial t}} X |_{N \times \{ \frac{\pi}{2} \}} = 0, 
\end{align*}
where $\nabla^{C_{s}(N)}$ and $\nabla^{N}$ are the Levi-Civita connections of 
$\tilde{k}$ on $C_{s}(N)$ and $k$ on $N$, respectively. 
\end{lem}

Denoting 
$
\nabla_{e_{i}}^{\top} e_{j} = \sum_{k = 1}^{3} \Gamma_{i j}^{k} e_{k}
$
for 
$
1 \leq i, j \leq 3, 
$
we see the following from the computations above and Lemma \ref{cross_NK}. 
\begin{lem}
\begin{align*}
(e_{i} \times \eta_{a}) 
= 
\left( 
\begin{array}{cccc}
\eta_{4}  & -\eta_{3}   & \eta_{2} & -\eta_{1} \\
\eta_{3} & \eta_{4}  & -\eta_{1}   & -\eta_{2} \\
-\eta_{2}  & \eta_{1} & \eta_{4}   & -\eta_{3} 
\end{array} 
\right), \qquad
\nabla^{\perp}_{e_{i}} \eta_{j} = 
\sum_{k = 1}^{3} (\epsilon_{ijk} + \Gamma_{i j}^{k}) \eta_{k}, 
\end{align*}
where $\epsilon_{ijk}$ is the permutation symbol and $1 \leq i, j \leq 3$.
\end{lem}

Via the identification 
$\mathfrak{X}(L) \oplus C^{\infty}(L) 
\ni 
(\sum_{j = 1}^{3} v_{j} e_{j}, f) 
\mapsto
\sum_{j = 1}^{3} v_{j} \eta_{j} + f \eta_{4}
\in C^{\infty}(L \times \{ \frac{\pi}{2} \}, \nu)$
where $v_{j}, f \in C^{\infty}(L)$, 
we can calculate as in the proof of Proposition \ref{deform_asso_SE}. 
\end{proof}

\section{Computation in the standard sphere $S^{7}$}

By Definition \ref{def on R7}, 
$\mathbb{C}^{4}$  admits a torsion-free ${\rm Spin}(7)$-structure $(\Phi_{0}, h_{0})$, 
which induces the nearly parallel $G_{2}$-structure $(\varphi, g)$ on $S^{7}$ 
by Lemma \ref{NPstr_onSE}. 
In this section, 
we study the deformation spaces of 
homogeneous associative submanifolds in $S^{7}$, 
and prove Theorem \ref{rigid_summary} and \ref{Lag_asso_rigid}.

\subsection{Classification of homogeneous associative submanifolds in $S^{7}$}

Mashimo \cite{Mashimo} classified homogeneous Lagrangian submanifolds in $S^{6}$. 
Applying this classification, 
Lotay \cite{Lotay3} classified homogeneous associative submanifolds in $S^{7}$. 

\begin{prop} [\cite{Lotay3, Mashimo}]  \label{Lotay_classification_asso}
Let $A$ be a connected associative 3-fold in $S^{7} \subset \mathbb{C}^{4}$ 
which is the orbit
of a closed 3-dimensional Lie subgroup of ${\rm Spin}(7)$. 
If $A$ does
not lie in a totally geodesic $S^{6}$, 
then, up to the ${\rm Spin}(7)$-action, $A$ is either
\begin{enumerate}
\item $A_{1} \cong T^{3}$ given by Example \ref{asso_A1}, 
\item $A_{2} \cong {\rm SU}(2)/\mathbb{Z}_{3}$, 
or $A_{3} \cong {\rm SU}(2)$ given by Example \ref{asso_A2 3}, 
\end{enumerate}

If  $A$ lies in a totally geodesic $S^{6}$, 
then, up to the $G_{2}$-action, $A$ is either 
\begin{enumerate}
\item the totally geodesic $S^{3} \cong {\rm SU}(2)$, 
\item $L_{1} \cong {\rm SU}(2)$ given by Example \ref{asso_L1},
\item $L_{2} \cong {\rm SO}(3) \cong {\rm SU}(2)/\mathbb{Z}_{2}$ given by Example \ref{asso_L2}, or
\item $L_{3} \cong {\rm SU}(2)/A^{*}_{4}$, or $L_{4} \cong {\rm SU}(2)/D^{*}_{3}$ given by Example \ref{asso_L3 4}. 
\end{enumerate}
\end{prop}

Note that the automorphism group of nearly parallel $S^{7}$ is ${\rm Spin}(7)$ 
and that of nearly K\"{a}hler $S^{6}$ is $G_{2}$.

\begin{example} \label{asso_A1}
Define the action of 
$T^{3} \cong  (\mathbb{R}/ 2 \pi \mathbb{Z})^{3}$ on $\mathbb{C}^{4}$ by 
\begin{align*}
(\theta_{1}, \theta_{2}, \theta_{3}) \cdot {}^t\! (z_{1}, z_{2}, z_{3}, z_{4})
= {}^t\! (e^{i \theta_{1}} z_{1}, e^{i \theta_{2}} z_{2}, e^{i \theta_{3}} z_{3},
   e^{-i (\theta_{1} + \theta_{2} +\theta_{3})} z_{4}), 
\end{align*}
where $\theta_{i} \in \mathbb{R}/ 2 \pi \mathbb{Z}$ and $z_{i} \in \mathbb{C}$. 
Then 
\begin{align*}
A_{1} := T^{3} \cdot \frac{1}{2} {}^t\! (1, 1, 1, i) \cong T^{3}
\end{align*}
is special Legendrian given in \cite{Harvey Lawson}.
\end{example}

\begin{example} \label{asso_A2 3}
Define the ${\rm SU}(2)$-action on $\mathbb{C}^{4}$ by 

\begin{align} \label{SU(2)actionC4}
\left( 
\begin{array}{cc}
a & -\overline{b} \\
b &   \overline{a} \\
\end{array} 
\right) 
\cdot 
\left( 
\begin{array}{c}
z_{1} \\
z_{2} \\
z_{3} \\
z_{4} \\
\end{array} 
\right) 
= 
\left( 
\begin{array}{cccc}
a^{3}               & -\sqrt{3} a^{2} \overline{b} & \sqrt{3} a \overline{b}^{2}        & - \overline{b}^{3}\\
\sqrt{3} a^{2}b  & a (|a|^{2} - 2 |b|^{2})         & - \overline{b} (2|a|^{2} - |b|^{2}) & \sqrt{3} \overline{a} \overline{b}^{2}\\
\sqrt{3} a b^{2} & b (2|a|^{2} - |b|^{2})          & \overline{a} (|a|^{2} - 2 |b|^{2})  &-\sqrt{3} \overline{a}^{2} \overline{b}\\
b^{3}               & \sqrt{3} \overline{a} b^{2}  & \sqrt{3} \overline{a}^{2} b        & \overline{a}^{3}
\end{array} 
\right)
\left( 
\begin{array}{c}
z_{1} \\
z_{2} \\
z_{3} \\
z_{4} \\
\end{array} 
\right), 
\end{align}
where $z_{i} \in \mathbb{C}$ and 
$a, b \in \mathbb{C}$ such that $|a|^{2} + |b|^{2} = 1$. 
Set 
\begin{align*}
A_{2} = {\rm SU}(2) \cdot 
                            {}^t\! (1, 0, 0, 0) \cong {\rm SU}(2)/ \mathbb{Z}_{3}, \qquad
A_{3} = {\rm SU}(2) \cdot \frac{1}{\sqrt{2}} {}^t\! (0, 1, i, 0) \cong {\rm SU}(2), 
\end{align*}
where 
$
\mathbb{Z}_{3} = 
\left \{ 
\left( 
\begin{array}{cc}
\zeta & 0 \\
0       &   \overline{\zeta} 
\end{array} 
\right) 
\in {\rm SU}(2); 
\zeta^{3} = 1
\right \}.
$
Then 
$A_{2}$ is the Hopf lift of the Veronese curve in $\mathbb{C}P^{3}$: 
\begin{align*}
\{ [x^{3} : \sqrt{3} x^{2} y : \sqrt{3} x y^{2} : y^{3}] \in \mathbb{C}P^{3}; [x : y] \in \mathbb{C}P^{1} \}, 
\end{align*}
and hence associative. 
However, $A_{3}$ is an associative submanifold which does not arise from other known geometries. 
\end{example}

\begin{rem} \label{congr_A2}
Set 
$A_{2}(\theta) = {\rm SU}(2) \cdot 
                            {}^t\! (\cos \theta, 0, 0, \sin \theta)$ 
for $\theta \in [0, \frac{\pi}{4}]$. 
It is known that 
all the $A_{2}(\theta)$ are congruent up to the ${\rm Spin}(7)$-action
to $A_{2} = A_{2}(0)$, which is ${\rm U}(2)$-invariant. 
In \cite{Joyce}, 
$A_{2}(\frac{\pi}{4})$ is shown to be special Legendrian. 
\end{rem}

Next, we give examples of homogeneous Lagrangian submanifolds in $S^{6}$.

\begin{example} \label{asso_L1}
Define the ${\rm SU}(2)$-action on $\mathbb{R}^{7} = \mathbb{R} \oplus \mathbb{C}^{3}$ by 
\begin{align} \label{SU2action_L1}
\left( 
\begin{array}{cc}
a & -\overline{b} \\
b &   \overline{a} \\
\end{array} 
\right) 
\cdot 
\left( 
\begin{array}{c}
x_{1} \\
z_{1} \\
z_{2} \\
z_{3} \\
\end{array} 
\right) 
= 
\left( 
\begin{array}{c}
(|a|^{2} - |b|^{2}) x_{1} - 2 {\rm Im}(\overline{ab} z_{1}) \\
2i \overline{a} b x_{1} + \overline{a}^{2} z_{1} + b^{2} \overline{z_{1}} \\
a z_{2} - \overline{b} \overline{z_{3}}  \\
\overline{b} \overline{z_{2}} + a z_{3} \\
\end{array} 
\right), 
\end{align}
where $a, b \in \mathbb{C}$ such that $|a|^{2} + |b|^{2} = 1$. 
Then 
\begin{align}
L_{1}:= {\rm SU}(2) \cdot {}^t\! (\tfrac{\sqrt{5}}{3}, 0, \tfrac{2}{3}, 0) \cong {\rm SU}(2), 
\end{align}
where ${}^t\! (\frac{\sqrt{5}}{3}, 0, \frac{2}{3}, 0) \in \mathbb{R} \oplus \mathbb{C}^{3}$, 
is Lagrangian in $S^{6}$ given in \cite{Harvey Lawson}. 
Moreover, $L_{1}$ is invariant under a ${\rm U}(2) (\subset G_{2})$ action.  
\end{example}

\begin{example} \label{asso_L2}
Let $L_{2} \subset S^{6}$ be given by 
\begin{align}
L_{2} =
\left \{
(0, z_{1}, z_{2}, z_{3}) \in \mathbb{R} \oplus \mathbb{C}^{3}; 
|z_{1}|^{2} + |z_{2}|^{2} + |z_{3}|^{2} = 1, 
z_{1}^{2} + z_{2}^{2} + z_{3}^{2} = 0
\right \}. 
\end{align}
Since $L_{2}$ is the link of an complex cone, it is Lagrangian in $S^{6}$. 
Define the ${\rm SO}(3)$-action on $\mathbb{R}^{7} = \mathbb{R} \oplus \mathbb{C}^{3}$ 
by the trivial action of  $\mathbb{R}$ and the standard (real) action on $\mathbb{C}^{3}$. 
Let $\varpi : {\rm SU}(2) \rightarrow {\rm SO}(3)$ 
be a standard double covering: 
\begin{align} \label{covering_SU2_SO3}
\varpi: 
\left( 
\begin{array}{cc}
a & -\overline{b} \\
b &   \overline{a} \\
\end{array} 
\right) 
\mapsto
\left( 
\begin{array}{ccc}
|a|^{2} - |b|^{2}                 & 2 {\rm Im}(ab)            & -2 {\rm Re}(ab)\\
-2 {\rm Im}(\overline{a} b) & {\rm Re}(a^{2} + b^{2})  & {\rm Im}(a^{2} + b^{2})\\
 2 {\rm Re}(\overline{a} b) & {\rm Im}(-a^{2} + b^{2}) & {\rm Re}(a^{2} - b^{2}) 
\end{array} 
\right), 
\end{align}
where $a, b \in \mathbb{C}$ such that $|a|^{2} + |b|^{2} = 1$. 
By composing these actions, 
${\rm SU}(2)$ acts on $\mathbb{R}^{7}$, and we have  
\begin{align*}
L_{2} = {\rm SU}(2) \cdot \frac{1}{\sqrt{2}} {}^t\! (0, 0, 1, i) \cong {\rm SU}(2)/ \mathbb{Z}_{2} = {\rm SO}(3). 
\end{align*}
\end{example}

\begin{example} \label{asso_L3 4}
Let $\{\epsilon_{1}, \cdots, \epsilon_{7} \}$ be a standard basis for $\mathbb{R}^{7}$. 
Identify $\mathbb{R}^{7}$ with the homogeneous harmonic cubics 
$\mathcal{H}^{3}(\mathbb{R}^{3})$ on $\mathbb{R}^{3}$ by: 
\begin{align*}
&\epsilon_{1} \mapsto 
\frac{ \sqrt{10}}{10} x (2 x^{2} -3 y^{2} - 3z^{2}); \qquad
&
&\epsilon_{2} \mapsto 
-\sqrt{6} xyz; \qquad
\epsilon_{3} \mapsto 
\frac{\sqrt{6}}{2} x (y^{2} - z^{2}); \\
&\epsilon_{4} \mapsto 
-\frac{\sqrt{15}}{10} y (4 x^{2} - y^{2} - z^{2}); \qquad
&
&\epsilon_{5} \mapsto 
-\frac{\sqrt{15}}{10} z (4 x^{2} - y^{2} - z^{2}); \\
&\epsilon_{6} \mapsto 
\frac{1}{2} y (y^{2} - 3 z^{2}); \qquad
&
&\epsilon_{7} \mapsto 
\frac{1}{2} z (z^{2} - 3 y^{2}). 
\end{align*}
Let ${\rm SU}(2)$ act on $\mathcal{H}^{3}(\mathbb{R}^{3}) \cong \mathbb{R}^{7}$ as 
$
A \cdot f(x, y, z) = f((x, y, z) \varpi(A)), 
$
where $A \in {\rm SU}(2)$ and $f \in \mathcal{H}^{3}(\mathbb{R}^{3}) \cong \mathbb{R}^{7}$. 
Set 
\begin{align*}
L_{3}:= {\rm SU}(2) \cdot \epsilon_{2}, \qquad
L_{4}:= {\rm SU}(2) \cdot \epsilon_{6}. 
\end{align*}
Then $L_{3} \cong {\rm SU}(2)/ A^{*}_{4}$ and 
$L_{4} \cong {\rm SU}(2)/D^{*}_{3}$ are Lagrangian, 
where 
$A^{*}_{4}$ is a binary tetrahedral group of order 24 generated by 
\begin{align} \label{generator A4}
k_{1}=\left( 
\begin{array}{cc}
i  & 0 \\
0 & -i
\end{array} 
\right), \qquad
k_{2}=\left( 
\begin{array}{cc}
0 & -1 \\
1 & 0
\end{array} 
\right), \qquad
k_{3}=
\frac{1}{\sqrt{2}}
\left( 
\begin{array}{cc}
e^{\frac{\pi i}{4}} & - e^{\frac{- \pi i}{4}} \\
e^{\frac{\pi i}{4}} & e^{\frac{- \pi i}{4}}
\end{array} 
\right),  
\end{align}
and $D^{*}_{3}$ is a binary dihedral group of order 12 generated by 
\begin{align} \label{generator D3}
k_{4}=
\left( 
\begin{array}{cc}
0 & -1 \\
1 & 0
\end{array} 
\right), \qquad
k_{5}=
\left( 
\begin{array}{cc}
e^{\frac{\pi i}{3}} & 0\\
0                     & e^{- \frac{\pi i}{3}}
\end{array} 
\right). 
\end{align}

\end{example}

%%%%%%%%%%%%%%%%%%%%%%%%%%%%%%%%%%%%%%%%%%%%%%%%%%%%%%%%%%%%%%%%%%%%%%%%%%%%%%%%%%%%%%%%%%

\subsection{Computations on ${\rm SU}(2)$}

For the convenience of the following computations, 
we summarize formulas on ${\rm SU}(2)$. 
Define the basis of the Lie algebra $\mathfrak{su}(2)$ of ${\rm SU}(2)$ by 
\begin{align} \label{E1E2E3}
E_{1} = 
\left( 
\begin{array}{cc}
0   & 1 \\
-1 & 0
\end{array} 
\right), \qquad
E_{2} = 
\left( 
\begin{array}{cc}
0 & i \\
i & 0
\end{array} 
\right), \qquad
E_{3} = 
\left( 
\begin{array}{cc}
i &   0 \\
0 & -i
\end{array} 
\right), 
\end{align}
which satisfies the relation $[E_{i}, E_{i + 1}] = 2 E_{i + 2}$ for  $i \in \mathbb{Z}/3$. 
We see the following by Proposition \ref{Fourier_general} and Lemma \ref{Schur_orthog}.

\begin{lem} \label{Fourier_S3}
Let 
$V_{n}$ be a $\mathbb{C}$-vector space 
of all complex homogeneous polynomials 
with two variables $z_{1}, z_{2}$ of degree $n (n \geq 0)$ and 
define the representation 
$\rho_{n} : {\rm SU}(2) \rightarrow {\rm GL}(V_{n})$ as 
\begin{align*} 
\left(\rho_{n} 
\left( 
\begin{array}{cc}
a & -\overline{b} \\
b &   \overline{a} \\
\end{array} 
\right) f
\right) (z_{1}, z_{2}) 
=
f
\left( 
(z_{1}, z_{2}) 
\left(
\begin{array}{cc}
a & -\overline{b} \\
b &   \overline{a} \\
\end{array} 
\right) 
\right).  
\end{align*} 
Define the Hermitian inner product $\langle \ , \ \rangle$ 
of $V_{n}$ such that 
\begin{align*} 
\left \{
v^{(n)}_{k}
=
\frac{1}{\sqrt{k ! (n - k)!}} z_{1}^{n-k} z_{2}^{k}
\right \}_{0 \leq k \leq n}  
\end{align*} 
is a unitary basis of $V_{n}$. 
If we denote by $\widehat{{\rm SU}(2)}$  
the set of all equivalence classes of finite dimensional irreducible representations of ${\rm SU}(2)$, 
we know that 
$\widehat{{\rm SU}(2)} = \{ (V_{n} ,\rho_{n}) ; n \geq 0 \}$.  
Then 
every $\mathbb{C}$-valued continuous function 
on ${\rm SU}(2)$ is 
uniformly approximated by 
the $\mathbb{C}$-linear combination of the following functions: 
\begin{align*}
\left \{ \langle \rho_{n} (\cdot) v^{(n)}_{i}, v^{(n)}_{j} \rangle ; 
n \geq 0, 0 \leq i, j \leq n
\right \}, 
\end{align*} 
which are mutually orthogonal with respect to the $L_{2}$ inner product. 
\end{lem}

By a direct computation, we see the following. 

\begin{lem} \label{calc_S3}
Identify $X \in \mathfrak{su}(2) \subset U(\mathfrak{su}(2))$ with 
the left invariant differential operator on ${\rm SU}(2)$. 
For $u = \sum_{l = 0}^{n} C_{l} v^{(n)}_{l} \in V_{n}$, 
set  
\begin{align*}
u^{*} = \sum_{l = 0}^{n} (-1)^{n - l} \overline{C}_{n - l} v^{(n)}_{l} \in V_{n}. 
\end{align*}
Then 
for any $n \geq 0, 0 \leq k, l \leq n,  u, v \in V_{n}, X \in \mathfrak{su}(2)$, we have 
\begin{align*}
X \langle \rho_{n} (\cdot) v, u \rangle
&=
\langle \rho_{n} (\cdot)  d \rho_{n}  (X) v, u \rangle,\\
(d \rho_{n} (X) v) (z_{1}, z_{2}) 
&= 
\left( \frac{\partial v}{\partial z_{1}}, \frac{\partial v}{\partial z_{2}} \right) {}^t\! X
\left(
\begin{array}{l}
z_{1} \\
z_{2}
\end{array}
\right), \\
\overline{ \langle \rho_{n} (\cdot) v^{(n)}_{k}, u \rangle} 
&=
(-1)^{k}
\langle \rho_{n} (\cdot) v^{(n)}_{n - k}, u^{*} \rangle,  \\
(-i E_{1} + E_{2}) 
\langle \rho_{n} (\cdot) v^{(n)}_{k}, u \rangle
&=
\left\{
\begin{array}{ll}
2i \sqrt{(k + 1)(n - k)} \langle \rho_{n} (\cdot) v^{(n)}_{k + 1}, u \rangle, & (k < n) \\
0,                                                                                                      & (k = n)
\end{array}
\right.\\
(i E_{1} + E_{2}) 
\langle \rho_{n} (\cdot) v^{(n)}_{k}, u \rangle
&=
\left\{
\begin{array}{ll}
2i \sqrt{k (n - k + 1)} \langle \rho_{n} (\cdot) v^{(n)}_{k - 1}, u \rangle, & (k > 0) \\
0,                                                                                                      & (k = 0)
\end{array}
\right.\\
i E_{3} \langle \rho_{n} (\cdot) v^{(n)}_{k}, u \rangle
&= (-n + 2k) \langle \rho_{n} (\cdot) v^{(n)}_{k}, u \rangle. \\ 
\end{align*}
\end{lem}

The next lemma is useful for the later computations.

\begin{lem} \label{gen_rot}
Suppose that $\{ e_{1}, e_{2}, e_{3} \} = \{ p E_{1}, p E_{2}, q E_{3} \}$
where $0 \neq p, q \in \mathbb{R}$ 
is an oriented orthonormal basis of $\mathfrak{su}(2)$ 
for some metric and orientation. 
For 
$v = \sum_{i = 1}^{3} v_{i} e_{i} \in C^{\infty} ({\rm SU}(2), \mathfrak{su}(2))$, 
$\overline{{\rm rot}} (v) = \alpha v$ for $0 \neq \alpha \in \mathbb{R}$ 
is equivalent to 
\begin{align}
\left( i e_{3} - (2q + \alpha) \right) (v_{1} + i v_{2}) + (-i e_{1} + e_{2}) v_{3} &= 0, \label{gen_rot1}\\
(i e_{1} + e_{2}) (v_{1} + i v_{2}) + \left( \alpha + \frac{2 p^{2}}{q} + i e_{3} \right) v_{3} &= 0. \label{gen_rot2}
\end{align}
These equations imply that 
\begin{align}
\left \{ \Delta_{+} + \left (\frac{4 p^{2}}{q} - 4q \right) i e_{3} 
+ \left( -\alpha - \frac{2 p^{2}}{q} + 2q \right) (2q + \alpha) \right \} (v_{1} + i v_{2}) = 0, \label{gen_rot3}\\
\left \{ \Delta_{+} - \alpha \left( \alpha + \frac{2 p^{2}}{q} \right) \right \} v_{3} = 0, \label{gen_rot4} 
\end{align}
where $\Delta_{+} = - \sum_{i = 1}^{3} e_{i}^{2}$ is a Laplacian. 
Especially, for any $n \geq 0, 0 \leq k \leq n, u \in V_{n}$, we have 
\begin{align}
&\Delta_{+} \langle \rho_{n} (\cdot) v^{(n)}_{k}, u \rangle
=
\left \{
(-p^{2} + q^{2})(n - 2k)^{2} 
+
p^{2} (n^{2} + 2n)
\right \}
\langle \rho_{n} (\cdot) v^{(n)}_{k}, u \rangle, \label{gen_rot5} \\
&\left \{ \Delta_{+} + \left (\frac{4 p^{2}}{q} - 4q \right) i e_{3} 
+ \left( -\alpha - \frac{2 p^{2}}{q} + 2q \right) (2q + \alpha) \right \} 
\langle \rho_{n} (\cdot) v^{(n)}_{k}, u \rangle  \nonumber \\
&=
\left \{
(-p^{2} + q^{2})(n - 2k + 2)^{2} 
+
p^{2} (n^{2} + 2n) - \alpha \left( \alpha + \frac{2 p^{2}}{q} \right) 
\right \}
\langle \rho_{n} (\cdot) v^{(n)}_{k}, u \rangle. \label{gen_rot6}
\end{align}
\end{lem}

\begin{rem}
In the case of ${\rm SU}(2)/ \Gamma$ for some finite subgroup $\Gamma$, 
we have to consider the $\Gamma$ equivariant solutions of (\ref{gen_rot1}) and (\ref{gen_rot2}). 
\end{rem}

\begin{proof}
Note that 
$
[e_{1}, e_{2}] = \frac{2p^{2}}{q} e_{3}, 
[e_{1}, e_{3}] = -2q e_{2}, 
[e_{2}, e_{3}] =  2q e_{1}.
$
Then from Remark \ref{rot_homog}, 
$\overline{{\rm rot}}(v) = \alpha v$ is equivalent to 
\begin{align}
i e_{3} (v_{1} + i v_{2}) + (-i e_{1} + e_{2})(v_{3}) &= (2q + \alpha) (v_{1} + i v_{2}),               \label{gen_rot7} \\
{\rm Re}((i e_{1} + e_{2}) (v_{1} + i v_{2}))        &= - \left( \alpha + \frac{2p^{2}}{q} \right) v_{3}. \label{gen_rot8}   
\end{align}
It is clear that (\ref{gen_rot1}) and (\ref{gen_rot2}) imply 
(\ref{gen_rot7}) and (\ref{gen_rot8}). 
Conversely, suppose that (\ref{gen_rot7}) and (\ref{gen_rot8}) hold. 
Applying $(i e_{1} + e_{2})$ to (\ref{gen_rot7}), we obtain 
\begin{align*}
(i e_{3} - \alpha) (i e_{1} + e_{2})(v_{1} + i v_{2}) 
+ 
\left( e_{1}^{2} + e_{2}^{2} + \frac{2p^{2}}{q} i e_{3} \right) v_{3} 
= 0.
\end{align*}
Considering the real and imaginary parts, we obtain from (\ref{gen_rot8}) 
\begin{align}
- e_{3} {\rm Im} ((i e_{1} + e_{2})(v_{1} + i v_{2})) 
+ \alpha \left(\alpha + \frac{2p^{2}}{q} \right) v_{3} 
+ (e_{1}^{2} + e_{2}^{2}) (v_{3}) = 0, \label{gen_rot9} \\
- \alpha e_{3}(v_{3}) 
- \alpha {\rm Im} ((i e_{1} + e_{2})(v_{1} + i v_{2})) = 0. \label{gen_rot10}
\end{align}
The equations (\ref{gen_rot8}) and (\ref{gen_rot10}) imply (\ref{gen_rot2}),  
and hence we obtain the first statement.

Substituting (\ref{gen_rot10}) into (\ref{gen_rot9}), 
we have (\ref{gen_rot4}). 
Applying $(-i e_{1} + e_{2})$ to (\ref{gen_rot2}), we obtain from (\ref{gen_rot1})
\begin{align*}
&\left( e_{1}^{2} + e_{2}^{2} - \frac{2p^{2}}{q} i e_{3} \right) (v_{1} + i v_{2}) \\
=& 
\left( - \alpha - \frac{2p^{2}}{q} +2q - i e_{3} \right) (-i e_{1} + e_{2})v_{3} \\
=& 
\left \{ -e_{3}^{2} + \left (\frac{4 p^{2}}{q} - 4q \right) i e_{3} 
+ \left( -\alpha - \frac{2 p^{2}}{q} + 2q \right) (2q + \alpha) \right \} (v_{1} + i v_{2}), 
\end{align*}
which imply (\ref{gen_rot3}). 
Then from Lemma \ref{calc_S3}, we obtain (\ref{gen_rot5}) and (\ref{gen_rot6}). 
\end{proof}

%%%%%%%%%%%%%%%%%%%%%%%%%%%%%%%%%%%%%%%%%%%%%%%%%%%%%%%%%%%%%%%%%%%%%%%%%%%%%%%%%%%%

%%%%%%%%%%%%%%%%%%%%%%%%%%%%%%%%%%%%%%%%%%%%%%%%%

\subsection{The case $A_{1}, A_{2}$, and $A_{3}$}

First, we study the deformation of 
homogeneous associative submanifolds 
which do not lie in a totally geodesic $S^{6}$.

\subsubsection{The case $A_{1} \cong T^{3}$} \label{T3_deform}

Define the basis of the Lie algebra $\mathfrak{t}^{3}$ of $T^{3}$ by 
\begin{align*}
e_{1} = (\sqrt{2}, 0, 0), \qquad 
e_{2} = (0 ,\sqrt{2}, -\sqrt{2}), \qquad 
e_{3} = (-1, 1, 1)
\in
\mathbb{R}^{3} \cong \mathfrak{t}^{3}, 
\end{align*}
which is an oriented orthonormal basis of $\mathfrak{t}^{3}$ 
with respect to the orientation and the metric 
induced from $A_{1}$.

Define the smooth function $f_{\gamma} \in C^{\infty}(T^{3}, \mathbb{C})$ 
for $\gamma = (\gamma_{1}, \gamma_{2}, \gamma_{3}) \in \mathbb{Z}^{3}$ 
on $T^{3} \cong  (\mathbb{R}/ 2 \pi \mathbb{Z})^{3}$ 
by 
$
f_{\gamma}(\theta_{1}, \theta_{2}, \theta_{3}) = \exp (i \sum_{j = 1}^{3} \gamma_{j} \theta_{j}).  
$
By a Fourier series expansion, 
every $\mathbb{C}$-valued continuous function 
on $T^{3}$ is 
uniformly approximated by 
the $\mathbb{C}$-linear combination of $f_{\gamma}$'s. 
By a direct computation, 
we obtain the following. 

\begin{lem} \label{calc_A1}
Identifying $e_{i} \in \mathfrak{t}^{3}$ with the left invariant differential operator on $T^{3}$, 
we have 
\begin{align*}
e_{1}(f_{\gamma}) &= \sqrt{2} \gamma_{1} i f_{\gamma}, \qquad 
e_{2}(f_{\gamma}) = \sqrt{2} (\gamma_{2} - \gamma_{3}) i f_{\gamma}, \qquad
e_{3}(f_{\gamma}) =            (-\gamma_{1} + \gamma_{2} + \gamma_{3}) i f_{\gamma}, \\
\Delta_{+}(f_{\gamma}) &= \{ 2 \gamma_{1}^{2} +  2 (\gamma_{2} - \gamma_{3})^{2} 
+ (-\gamma_{1} + \gamma_{2} + \gamma_{3})^{2} \}  f_{\gamma}.
\end{align*}
\end{lem}

Then we deduce the following.

\begin{prop} \label{sol_lap_A1}
$\dim_{\mathbb{R}} \{ f \in C^{\infty}(T^{3}) ; \Delta_{+} f = 8f \} = 12$. 
\end{prop}

\begin{prop} \label{sol_rot_A1}
$\dim_{\mathbb{R}} \{ v \in C^{\infty}( T^{3}, \mathfrak{t}^{3}) ; 
 \overline{{\rm rot}}(v) = - 2v \} = 6$. 
\end{prop}
By Corollary \ref{diff_asso_sL}, 
these imply that associative deformations of $A_{1}$ are trivial since 
${\rm Spin(7)}$ induces 
$18 (= \dim_{\mathbb{R}} ({\rm Spin}(7)/ T^{3}))$-dimensional 
associative deformations of $A_{1}$. 
Now, we give proofs.

\begin{proof}[Proof of Proposition \ref{sol_lap_A1}]
For $(\gamma_{1}, \gamma_{2}, \gamma_{3}) \in \mathbb{Z}^{3}$, we know that 
\begin{align*}
&2 \gamma_{1}^{2} +  2 (\gamma_{2} - \gamma_{3})^{2} 
+ (-\gamma_{1} + \gamma_{2} + \gamma_{3})^{2} = 8 \\
\Leftrightarrow 
&(\gamma_{1}, \gamma_{2}, \gamma_{3}) = \pm(2, 1, 1), \pm(0, 1, -1), 
                                                       \pm(1, 2, 1), \pm(1, 1, 2), 
                                                        \pm(1, -1, 0), \pm(1, 0, -1), 
\end{align*}
which gives the proof by Lemma \ref{calc_A1}. 
\end{proof}

\begin{proof}[Proof of Proposition \ref{sol_rot_A1}]
Take any $v = \sum_{i = 1}^{3} v_{i} e_{i} \in C^{\infty}(T^{3}, \mathfrak{t}^{3})$ 
where $v_{i} \in C^{\infty}(T^{3})$. 
Then from Remark \ref{rot_homog}, 
$\overline{{\rm rot}}(v) = \alpha v$ for $\alpha \in \mathbb{R}$ is equivalent to 
\begin{align}
(i e_{3} - \alpha) (v_{1} + i v_{2}) + (-i e_{1} + e_{2})(v_{3}) &= 0,  \label{rot_A1_1} \\
{\rm Re}((i e_{1} + e_{2}) (v_{1} + i v_{2}))        &= - \alpha v_{3}.              \label{rot_A1_2}   
\end{align}
Eliminating $v_{3}$, we have 
\begin{align} \label{rot_A1_3}
-\alpha (i e_{3} - \alpha) (v_{1} + i v_{2}) + (- i e_{1} + e_{2}) {\rm Re}((i e_{1} + e_{2})(v_{1} + i v_{2})) = 0. 
\end{align}
Set 
$v_{1} + i v_{2} = \sum_{\gamma \in \mathbb{Z}^{3}} C_{\gamma} f_{\gamma}$ 
where $C_{\gamma} \in \mathbb{C}$. 
Since $\overline{f}_{\gamma} = f_{-\gamma}$, 
(\ref{rot_A1_3}) is equivalent to 
\begin{align} \label{rot_A1_4}
C_{\gamma} 
\left( - \gamma_{1}^{2} - (\gamma_{2} - \gamma_{3})^{2} + 
\alpha (- \gamma_{1} + \gamma_{2} + \gamma_{3} + \alpha) 
\right) 
+ 
\overline{C}_{-\gamma} (\gamma_{1} + (\gamma_{2} - \gamma_{3})i)^{2} = 0. 
\end{align}
Take the complex conjugation of (\ref{rot_A1_4}) and 
replace $\gamma$ by $- \gamma$, then we obtain 
\begin{align} \label{rot_A1_5}
C_{\gamma} 
 (\gamma_{1} + (- \gamma_{2} + \gamma_{3})i)^{2} 
+ 
\overline{C}_{-\gamma}
\left( - \gamma_{1}^{2} - (\gamma_{2} - \gamma_{3})^{2} + 
\alpha (\gamma_{1} - \gamma_{2} - \gamma_{3} + \alpha) 
\right) 
= 0. 
\end{align}
Eliminating $\overline{C}_{\gamma}$ from (\ref{rot_A1_4}) and (\ref{rot_A1_5}), we have 
\begin{align*}
\alpha^{2} 
\left \{
-2 
(\gamma_{1}^{2} + (\gamma_{2} - \gamma_{3})^{2}) 
- \alpha^{2}
(- \gamma_{1} + \gamma_{2} + \gamma_{3})^{2} 
+ \alpha^{2}
\right \} 
C_{\gamma} 
= 0.  
\end{align*}

Set $\alpha = -2$. 
Since we know that 
$-2 
(\gamma_{1}^{2} + (\gamma_{2} - \gamma_{3})^{2}) 
- 
(- \gamma_{1} + \gamma_{2} + \gamma_{3})^{2} 
+ 4
= 0 
\Leftrightarrow 
(\gamma_{1},  \gamma_{2}, \gamma_{3})
= 
\pm (1, 1, 0), \pm (1, 0, 1), \pm (0, 1, 1), 
$
we deduce by (\ref{rot_A1_4}) that 
\begin{align*}
v_{1} + i v_{2} =& C_{(1, 1, 0)} f_{(1, 1, 0)} -i \overline{C}_{(1, 1, 0)} f_{(-1, -1, 0)} 
                     +C_{(1, 0, 1)} f_{(1, 0, 1)} +i \overline{C}_{(1, 0, 1)} f_{(-1, 0, -1)} \\
                   &+C_{(0, -1, -1)} f_{(0, -1, -1)}.  
\end{align*}
Thus $v_{1} + i v_{2}$ depends 3 complex parameters $C_{(1, 1, 0)}, C_{(1, 0, 1)}, C_{(0, -1, -1)}$, 
which implies Proposition \ref{sol_rot_A1}. 
\end{proof}

%%%%%%%%%%%%%%%%%%%%%%%%%%%%%%%%%%%%%%%%%%%%%%%%%%%%%%%%%%%%%%%%%%%%%%%%%%%%%%%%%

\subsubsection{The case $A_{2} \cong {\rm SU}(2)/ \mathbb{Z}_{3} $} \label{A2_deform}

By Remark \ref{congr_A2}, 
$A_{2} = A_{2}(0)$ is congruent to $A_{2}(\frac{\pi}{4})$, which is special Legendrian. 
We may compute the dimension of the infinitesimal associative deformations of $A_{2}(\frac{\pi}{4})$ by 
Corollary \ref{diff_asso_sL}. 
The action (\ref{SU(2)actionC4}) induces 
an inclusion 
$\mathfrak{su}(2) \hookrightarrow \mathfrak{su}(4)$,  
where 
$E_{1}, E_{2}, E_{3}$ in (\ref{E1E2E3})
correspond to 
\begin{align*}
\left( 
\begin{array}{cccc}
0           & \sqrt{3} &0         & 0\\
-\sqrt{3} & 0         &2          & 0 \\
0           &-2         &0          & \sqrt{3}\\
0           &0           &-\sqrt{3}&0
\end{array} 
\right), 
\left( 
\begin{array}{cccc}
0           & \sqrt{3}i &0           & 0\\
\sqrt{3}i  & 0          &2i          & 0 \\
0           &2i          &0           & \sqrt{3}i\\
0           &0           &\sqrt{3}i &0
\end{array} 
\right), 
\left( 
\begin{array}{cccc}
3i &0 &0 & 0\\
0 & i  &0 & 0 \\
0 &0  &-i &0 \\
0 &0  &0 &-3i
\end{array} 
\right),
\end{align*}
respectively. 
Set $p_{0} = \frac{1}{\sqrt{2}} {}^t\! (1, 0, 0, 1) \in \mathbb{C}^{4}$. 
Then we have 
\begin{align*}
(E_{1}^{*})_{p_{0}} = \sqrt{\frac{3}{2}}  {}^t\! (0, -1, 1, 0), \qquad
(E_{2}^{*})_{p_{0}} = \sqrt{\frac{3}{2}}  {}^t\! (0, i, i, 0),  \qquad
(E_{3}^{*})_{p_{0}} = \frac{3i}{\sqrt{2}}  {}^t\! (1, 0, 0, -1),  
\end{align*}
Hence if we set $e_{1} := E_{1}/ \sqrt{3}, e_{2} := E_{2}/ \sqrt{3}, e_{3} := E_{3}/ 3$, 
$\{ e_{i} \}_{1 \leq i \leq 3}$ is an oriented orthonormal basis of $\mathfrak{su}(2)$ 
with respect to the orientation and the metric 
induced from $A_{2}$. 

\begin{prop} \label{sol_lap_A2}
$\dim_{\mathbb{R}} \{ f \in C^{\infty}(A_{2}) ; \Delta_{+} f = 8f \} = 19$. 
\end{prop}

\begin{prop} \label{sol_rot_A2}
$\dim_{\mathbb{R}} \{ v \in \mathfrak{X} (A_{2});  {\rm rot}(v) = -2v \} = 11$. 
\end{prop}
On the other hand, 
${\rm Spin(7)}$ induces 
$17 (= \dim_{\mathbb{R}} ({\rm Spin}(7)/ {\rm U} (2)))$-dimensional 
associative deformations of $A_{2}$. 
By Corollary \ref{diff_asso_sL} and 
Remark \ref{congr_A2}, 
we have a $30$-dimensional 
infinitesimal associative deformation space of $A_{2}$, 
and hence 
$A_{2}$ can have non-trivial associative deformations. 
In fact, 
we obtain the following. 

\begin{prop} \label{nontrivial_A2}
All non-trivial associative deformations of $A_{2}$ are induced by 
the ${\rm PGL}(4, \mathbb{C})$-action on $\mathbb{C}P^{3}$ via the Hopf lift. 
\end{prop}

\begin{rem} (\cite{Ohnita_def, Bolton})
As a special Legendrian submanifold, 
$A_{2}(\frac{\pi}{4})$ is not rigid, either. 
By a non-standard projection $p_{2} : S^{7} \rightarrow \mathbb{C}P^{3}$, 
$p_{2}(A_{2}(\frac{\pi}{4}))$ is a horizontal holomorphic curve in $\mathbb{C}P^{3}$, 
and 
for any horizontal holomorphic curve $\Sigma$, 
$p_{2}^{-1}(\Sigma) \subset S^{7}$ is a special Legendrian submanifold. 

Since 
the group of biholomorphic maps which preserve 
the horizontal distribution is 
${\rm PSp}(2, \mathbb{C})$, 
all non-trivial special Legendrian deformations of $A_{2}(\frac{\pi}{4})$ 
are given by the induced action of ${\rm PSp}(2, \mathbb{C})$ on $\mathbb{C}P^{3}$. 
\end{rem}

Now, we give proofs. 
First, we prove the following lemma. 

\begin{lem} \label{Hom_A2}
Let $\{(v^{(n)}_{l})^{*} = \langle \cdot, v^{(n)}_{l} \rangle \}$ 
be the dual basis of $\{v^{(n)}_{l} \}$. Then we have 
\begin{align*}
&{\rm Hom}_{\mathbb{Z}_{3}} (V_{n}, \mathfrak{su}(2) \otimes_{\mathbb{R}} \mathbb{C}) \\
=& 
\left \{ 
L \in {\rm Hom}_{\mathbb{C}} (V_{n}, \mathfrak{su}(2) \otimes_{\mathbb{R}} \mathbb{C}) ; 
L(\rho_{n}(k) v) = 
{\rm Ad}
(k) L(v)
\mbox{ for any } k \in \mathbb{Z}_{3}, v \in V_{n}
\right \} \\
=&
{\rm span}_{\mathbb{C}} \left \{ (v^{(n)}_{i})^{*} \otimes X ; 
X = 
\left \{
\begin{array}{ll}
e_{3}            & (n - 2i \in 3 \mathbb{Z}) \\
e_{1} + i e_{2} & (n - 2i \in 3 \mathbb{Z} + 1) \\
e_{1} - i e_{2} & (n - 2i \in 3 \mathbb{Z} + 2) \\
\end{array}
\right.
\right \}, \\
&{\rm Hom}_{\mathbb{Z}_{3}} (V_{n}, \mathbb{C}) 
= {\rm span}_{\mathbb{C}} \left \{ (v^{(n)}_{i})^{*}; n - 2i \in 3 \mathbb{Z} \right \}.
\end{align*}
\end{lem}

\begin{proof}
Take any $v \in V_{n}$ and 
$k = 
\left( 
\begin{array}{cc}
\zeta & 0 \\
0       &   \overline{\zeta} 
\end{array} 
\right) 
\in \mathbb{Z}_{3}
$ 
where $\zeta^{3} = 1$. 
By definition, we see that 
\begin{align*}
{\rm Ad}(k) e_{1} &= {\rm Re}(\zeta) e_{1} - {\rm Im}(\zeta) e_{2}, \\
{\rm Ad}(k) e_{2} &= {\rm Im}(\zeta) e_{1} + {\rm Re}(\zeta) e_{2}, \\
{\rm Ad}(k) e_{3} &= e_{3}, \\
\rho_{n}(k) v^{(n)}_{l} &= \zeta^{n - 2l} v^{(n)}_{l}. 
\end{align*}
Setting 
$L = \sum_{l = 0}^{n} \sum_{i = 1}^{3} C_{l i} (v^{(n)}_{l})^{*} \otimes e_{i} 
\in {\rm Hom}_{\mathbb{C}}(V_{n}, \mathfrak{su}(2) \otimes_{\mathbb{R}} \mathbb{C} )$
where $C_{k i} \in \mathbb{C}$, 
we know that 
$L \in {\rm Hom}_{\mathbb{Z}_{3}}(V_{n}, \mathfrak{su}(2) \otimes_{\mathbb{R}} \mathbb{C} )$
if and only if 
\begin{align*}
\zeta^{n - 2l}  \sum_{i = 1}^{3} C_{l i} e_{i} 
= 
C_{l 1} ({\rm Re}(\zeta) e_{1} - {\rm Im}(\zeta) e_{2}) 
+ C_{l 2} ({\rm Im}(\zeta) e_{1} + {\rm Re}(\zeta) e_{2}) 
+ C_{l 3} e_{3} 
\end{align*}
for any $0 \leq l \leq n$ and $\zeta^{3} = 1$. 
This is equivalent to 
\begin{align*}
(\zeta^{n - 2l} - 1)C_{l 3} = 0, \\
(\zeta^{n - 2l + 1} - 1)(C_{l 2} -i C_{l 1}) = 0, \\
(\zeta^{n - 2l + 2} - 1)(C_{l 2} +i C_{l 1}) = 0, 
\end{align*}
which implies the first statement. 
The second is proven in the same way.  
\end{proof}

\begin{proof}[Proof of Proposition \ref{sol_lap_A2}]
From (\ref{gen_rot5}), 
the solution $f$ of $\Delta_{+} f = 8f$ is contained in 
$
{\rm span}_{\mathbb{C}} 
\left \{ 
\langle 
\rho_{n} (\cdot) 
v^{(n)}_{k}, 
v^{(n)}_{l}  
\rangle; 
(n, k) = (6, 0), (6, 6), (4, 2), 
0 \leq l \leq n
\right \}, 
$
which are $\mathbb{Z}_{3}$ invariant. 
Hence we obtain Proposition \ref{sol_lap_A2}. 
\end{proof}

\begin{proof}[Proof of Proposition \ref{sol_rot_A2}]
First,  we consider $\dim_{\mathbb{R}} \{ v \in \mathfrak{X} (A_{2}); \overline{{\rm rot}}(v) = -2 v \}$. 
Set $(p, q, \alpha) = (\frac{1}{\sqrt{3}}, \frac{1}{3}, -2)$ in Lemma \ref{gen_rot}. 
Since we know that 
\begin{align*}
-\frac{2}{9} (n -2k + 2)^{2} + \frac{1}{3} (n^{2} + 2n)  = 0 
\Leftrightarrow 
(n, k) = (4, 0), \\
-\frac{2}{9} (n -2k)^{2} + \frac{1}{3} (n^{2} + 2n) = 0 
\Leftrightarrow 
(n, k) = (0, 0), 
\end{align*}
we have 
$
v_{1} + i v_{2} =  \langle \rho_{4} (\cdot) v^{(4)}_{0}, u \rangle
$
for $u \in V_{4}$ and $v_{3}$ is constant. 
We see that 
$v = \sum_{i = 1}^{3} v_{i} e_{i}$ satisfies (\ref{gen_rot1}), (\ref{gen_rot2}), 
and is $\mathbb{Z}_{3}$ equivariant. 
Hence we obtain 
$\dim_{\mathbb{R}} \{ v \in \mathfrak{X} (A_{2}); {\rm rot}(v) = -2 v \} = 11$. 
\end{proof}

\begin{proof}[Proof of Proposition \ref{nontrivial_A2}]

We find $13(= 30 -17)$-dimensional family of non-trivial associative deformations. 

Let $p_{1} : S^{7} \rightarrow \mathbb{C}P^{3}$ be the Hopf fibration. 
By Lemma \ref{sLeg_cpx_asso}, for any holomorphic curve $\Sigma \subset \mathbb{C}P^{3}$, 
the Hopf lift $p_{1}^{-1}(\Sigma) \subset S^{7}$ of $\Sigma$ 
is an associative submanifold. 
Since $p_{1}(A_{2})$ is a holomorphic curve in $\mathbb{C}P^{3}$, 
the group of biholomorphic map of $\mathbb{C}P^{3}$, which is known to be ${\rm PGL}(4, \mathbb{C})$, 
induces the associative deformations of $A_{2}$ via the Hopf lift. 

The ${\rm PGL}(4, \mathbb{C})$-action included in the ${\rm Spin}(7)$-action 
is the standard ${\rm SU}(4)$-action on $S^{7}$. 
Thus 
the dimension of non-trivial associative deformations of $A_{2}$ 
induced by ${\rm PGL}(4, \mathbb{C})$
is given by 
\begin{align*}
&\dim_{\mathbb{R}} {\rm PGL}(4, \mathbb{C}) 
- 
\dim_{\mathbb{R}} \{ g \in {\rm PGL}(4, \mathbb{C}) ; g \cdot p_{1} (A_{2}) \subset p_{1} (A_{2}) \} \\
&- 
\left (
\dim_{\mathbb{R}} {\rm SU}(4) 
- \dim_{\mathbb{R}} \{ h \in {\rm SU}(4) ; h \cdot A_{2} \subset A_{2} \} 
\right )\\
=& \dim_{\mathbb{R}} {\rm PGL}(4, \mathbb{C}) - \dim_{\mathbb{R}} {\rm PGL}(2, \mathbb{C}) 
- \dim_{\mathbb{R}} {\rm SU}(4) + \dim_{\mathbb{R}} {\rm U}(2) \\
=&
30 - 6 - 15 + 4
= 13, 
\end{align*}
which gives the proof. 
\end{proof}

%%%%%%%%%%%%%%%%%%%%%%%%%%%%%%%%%%%%%%%%%%%%%%%%%%%%%%%%%%%%%%%%%%%%%%%%%%

\subsubsection{The case $A_{3} \cong {\rm SU}(2)$} \label{A3_deform}

Since $A_{3}$ is not special Legendrian, we cannot apply Corollary \ref{diff_asso_sL} to this case. 
First, we describe the operator $D$ explicitly. 
Define $E_{i} \in \mathfrak{su}(2)$ 
as (\ref{E1E2E3}). 
We denote by $e_{1}, e_{2}, e_{3}$
the left invariant vector fields 
on ${\rm SU}(2) \cong A_{3}$ 
induced by 
$\frac{1}{\sqrt{7}} E_{1}, \frac{1}{\sqrt{7}} E_{2}, E_{3}$, respectively. 
If we define the vectors 
$\eta_{k}$ for $1 \leq k \leq 4$ as 
\begin{align*}
\eta_{1} = \sqrt{\frac{7}{3}} \left(J e_{1} + \frac{2}{\sqrt{7}} e_{4} \right), \qquad
\eta_{2} = \sqrt{\frac{7}{3}} \left(J e_{2} + \frac{2}{\sqrt{7}} e_{3} \right), \\
\eta_{3} = \sqrt{\frac{7}{3}} \left(J e_{3} - \frac{2}{\sqrt{7}} e_{2} \right), \qquad
\eta_{4} = \sqrt{\frac{7}{3}} \left(J e_{4} - \frac{2}{\sqrt{7}} e_{1} \right), 
\end{align*}
where $J$ is the standard complex structure on $\mathbb{C}^{4}$ 
and $e_{4}$ is the position vector, 
then 
$\{ e_{1}, \cdots, e_{3} \}$ is the orthonormal frame of $TA_{3}$ and 
$\{ \eta_{1}, \cdots, \eta_{4} \}$ 
is the orthonormal frame of $\nu$. 
At $p_{0} = \frac{1}{\sqrt{2}} {}^t\! (0, 1, i, 0)$, we have 
\begin{align*}
e_{1} = 
\frac{1}{\sqrt{14}} 
\left( 
\begin{array}{c}
\sqrt{3} \\
2i         \\
-2        \\
-\sqrt{3}i
\end{array} 
\right), 
e_{2} = 
\frac{1}{\sqrt{14}} 
\left( 
\begin{array}{c}
\sqrt{3}i \\
-2         \\
2i        \\
-\sqrt{3}
\end{array} 
\right), 
e_{3} = 
\frac{1}{\sqrt{2}} 
\left( 
\begin{array}{c}
0 \\
i  \\
1\\
0
\end{array} 
\right), 
e_{4} = 
\frac{1}{\sqrt{2}} 
\left( 
\begin{array}{c}
0 \\
1 \\
i  \\
0
\end{array} 
\right), 
\end{align*}
\vspace{-0.3cm}
\begin{align*}
\eta_{1} 
= \frac{1}{\sqrt{2}} 
\left( 
\begin{array}{c}
i \\
0  \\
0 \\
1
\end{array} 
\right), 
\eta_{2} 
= \frac{1}{\sqrt{2}} 
\left( 
\begin{array}{c}
-1 \\
0  \\
0 \\
-i
\end{array} 
\right), 
\eta_{3} 
= \frac{1}{\sqrt{42}} 
\left( 
\begin{array}{c}
-2 \sqrt{3}i \\
-3  \\
3i \\
2 \sqrt{3}
\end{array} 
\right), 
\eta_{4} 
= \frac{1}{\sqrt{42}} 
\left( 
\begin{array}{c}
-2 \sqrt{3} \\
3 i  \\
-3 \\
2 \sqrt{3} i
\end{array} 
\right).
\end{align*}

\begin{lem}
We have 
\begin{align*}
\nabla_{e_{i}}^{\top} e_{i} = 0  \mbox{ for } i = 1, 2, 3, \qquad
[e_{1}, e_{2}] = \frac{2}{7} e_{3}, \qquad
[e_{1}, e_{3}] = -2 e_{2}, \qquad
[e_{2}, e_{3}] = 2 e_{1}, 
\end{align*}
\vspace{-0.5cm}
\begin{align*}
(\nabla^{\perp}_{e_{i}} \eta_{j}) 
=
\frac{3}{7}
\left( 
\begin{array}{cccc}
- \eta_{4} & -\eta_{3}   & \eta_{2}    & \eta_{1} \\
 \eta_{3}  & - \eta_{4}  & - \eta_{1}  & \eta_{2} \\
7\eta_{2}  & -7 \eta_{1} & -5 \eta_{4} & 5 \eta_{3} 
\end{array} 
\right), 
(e_{i} \times \eta_{j}) 
= 
\left( 
\begin{array}{cccc}
\eta_{4} & \eta_{3}  & -\eta_{2} & -\eta_{1} \\
-\eta_{3} & \eta_{4}   & \eta_{1}  & -\eta_{2} \\
\eta_{2}  & -\eta_{1} & \eta_{4}   & -\eta_{3} 
\end{array} 
\right).
\end{align*}
\end{lem}
\begin{proof}
Since the ${\rm SU}(2)$-action preserves the $G_{2}$-structure on $S^{7}$, 
we only have to consider at $p_{0}$. 
The equations of $\nabla_{e_{i}}^{\top} e_{i}$ and $[e_{i}, e_{j}]$
is shown easily. 
By a direct computation, we have 
\begin{align*}
(\nabla^{\mathbb{C}^{4}}_{e_{i}} \eta_{j}) 
=
\frac{3}{7}
\left( 
\begin{array}{cccc}
- \eta_{4} & -\eta_{3}   & \eta_{2}    & \eta_{1} \\
 \eta_{3}  & - \eta_{4}  & - \eta_{1}  & \eta_{2} \\
7\eta_{2}  & -7 \eta_{1} & -5 \eta_{4} & 5 \eta_{3} 
\end{array} 
\right) 
+
\frac{2 \sqrt{3}}{7}
\left( 
\begin{array}{cccc}
-e_{1} & -e_{2}  & 2 e_{3} & 0 \\
e_{2}  & -e_{1}  & 0        & -2 e_{3}\\
0      & 0        & 2 e_{1} & -2 e_{2}
\end{array} 
\right),  
\end{align*}
and hence we obtain $\nabla^{\perp}_{e_{i}} \eta_{j}$. 
To prove the equations of $e_{i} \times \eta_{j}$, 
let $h_{0}$ be the standard metric on $\mathbb{C}^{4}$, 
$\omega_{0}$ be the standard K\"{a}hler form on $\mathbb{C}^{4}$, 
and $\Omega_{0}$  be the standard holomorphic volume form on $\mathbb{C}^{4}$. 
Define 
$e^{i} = h_{0} (e_{i}, \cdot), \eta^{j} = h_{0}(\eta_{j}, \cdot)$. 
Then 
$\{ e^{1}, \cdots, e^{4}, \eta^{1}, \cdots, \eta^{4} \}$ 
is the dual coframe of 
$\{ e_{1}, \cdots, e_{4}, \eta_{1}, \cdots, \eta_{4} \}$. 
We compute 
\begin{align*}
\left( 
\begin{array}{c}
e^{1}(J \cdot) \\
e^{2}(J \cdot) \\
e^{3}(J \cdot) \\
e^{4}(J \cdot)
\end{array} 
\right) 
=
\frac{2}{\sqrt{7}}
\left( 
\begin{array}{c}
e^{4} \\
e^{3} \\
-e^{2} \\
-e^{1}
\end{array} 
\right) 
-
\sqrt{\frac{3}{7}}
\left( 
\begin{array}{c}
\eta^{1} \\
\eta^{2} \\
\eta^{3} \\
\eta^{4} 
\end{array} 
\right), 
\left( 
\begin{array}{c}
\eta^{1}(J \cdot) \\
\eta^{2}(J \cdot) \\
\eta^{3}(J \cdot) \\
\eta^{4}(J \cdot)
\end{array} 
\right) 
=
\sqrt{\frac{3}{7}}
\left( 
\begin{array}{c}
e^{1} \\
e^{2} \\
e^{3} \\
e^{4}
\end{array} 
\right) 
+ 
\frac{2}{\sqrt{7}}
\left( 
\begin{array}{c}
-\eta^{4} \\
-\eta^{3} \\
\eta^{2} \\
\eta^{1} 
\end{array} 
\right). 
\end{align*}
Since we know 
$h_{0} = \sum_{i = 1}^{4} ((e^{i})^{2} + (\eta^{i})^{2})$, we obtain 
\begin{align*}
\omega_{0} 
= h_{0}(J \cdot, \cdot) 
=
\sqrt{\frac{3}{7}} \sum_{i = 1}^{4} e^{i} \wedge \eta^{i} 
+ \frac{2}{\sqrt{7}} (-e^{14} - e^{23} + \eta^{14} + \eta^{23}).  
\end{align*}
The holomorphic volume form $\Omega_{0}$ 
is of the form 
$C \cdot (e^{1} + i g(e_{1}, \cdot)) \wedge \cdots (e^{4} + i g(e_{4}, \cdot)) 
=
C \cdot (e^{1} - i e^{1}(J \cdot)) \wedge \cdots (e^{4} - i e^{4}(J \cdot)) 
$
for $C>0$, 
and from the relation
$\omega_{0}^{4}/4! = (i/2)^{4} \Omega_{0} \wedge \overline{\Omega_{0}}$, 
we have $C=7/3$. 
Hence 
the $G_{2}$-structure $\varphi \in \Omega^{3}(S^{7})$ on $S^{7}$ is described as 
\begin{align*}
\varphi &= i(e_{4}) \left(\frac{1}{2} \omega_{0}^{2} + {\rm Re} \Omega_{0} \right) \\
      &= -e^{123} 
        +e^{1} \wedge (\eta^{14} + \eta^{23}) 
        +e^{2} \wedge (-\eta^{13} + \eta^{24}) 
        +e^{3} \wedge (\eta^{12} + \eta^{34}), 
\end{align*}
which implies the lemma. 
\end{proof}

\begin{prop} \label{explicit_D_A3}
By the trivialization of $\nu$ via $\{ \eta_{1}, \cdots, \eta_{4} \}$, 
$D :  C^{\infty}({\rm SU}(2), \mathbb{R}^{4}) \cong C^{\infty} (A_{3}, \nu) 
      \rightarrow C^{\infty} (A_{3}, \nu) \cong C^{\infty}({\rm SU}(2), \mathbb{R}^{4})$ 
is described as follows: 
\begin{align*}
D
\left( 
\begin{array}{c}
\psi_{1} \\
\psi_{2} \\
\psi_{3} \\
\psi_{4}
\end{array} 
\right) 
= 
\left \{
\left( 
\begin{array}{cccc}
0        & -e_{3}  & e_{2}   & -e_{1} \\
e_{3}    & 0       & -e_{1} & -e_{2} \\
-e_{2}  & e_{1}   & 0       & -e_{3} \\
e_{1}   & e_{2}   & e_{3}   & 0
\end{array} 
\right) 
+
\left( 
\begin{array}{cccc}
-\frac{15}{7}   &                  &    &  \\
                    & -\frac{15}{7}&    & \\
                    &                  & 3 & \\
                    &                  &   & 3
\end{array} 
\right) 
\right \}
\left( 
\begin{array}{c}
\psi_{1} \\
\psi_{2} \\
\psi_{3} \\
\psi_{4}
\end{array} 
\right). 
\end{align*}
Setting $\Psi_{1} = \psi_{1} + i \psi_{2}$, 
$\Psi_{2} = \psi_{3} - i \psi_{4}$, we have
\begin{align*}
D 
\left( 
\begin{array}{c}
\Psi_{1} \\
\Psi_{2} 
\end{array} 
\right)
=
\left \{
\left( 
\begin{array}{cc}
ie_{3}                & -i e_{1} + e_{2} \\
-(i e_{1} + e_{2})  & - i e_{3}          
\end{array} 
\right) 
+
\left( 
\begin{array}{cc}
-\frac{15}{7}  &  \\
                  & 3       
\end{array} 
\right) 
\right \}
\left( 
\begin{array}{c}
\Psi_{1} \\
\Psi_{2} 
\end{array} 
\right).
\end{align*}
\end{prop}

\begin{proof}
Take $\psi = \sum_{a = 1}^{4} \psi_{a} \eta_{a} \in C^{\infty} (A_{3}, \nu)$ for $\psi_{a} \in C^{\infty}(A_{3})$. 
By the lemma above, we see 
\begin{align*}
D \psi &= 
\sum_{i ,a} (e_{i} (\psi_{a}) e_{i} \times \eta_{a} + \psi_{a} e_{i} \times \nabla^{\perp}_{e_{i}} \eta_{a}) \\
&=
(- e_{3} (\psi_{2}) +e_{2} (\psi_{3}) - e_{1} (\psi_{4}) -\tfrac{15}{7} \psi_{1} ) \eta_{1} 
+ (e_{3} (\psi_{1}) -e_{1} (\psi_{3}) - e_{2} (\psi_{4}) -\tfrac{15}{7} \psi_{2} ) \eta_{2} \\
&+ ( - e_{2} (\psi_{1}) + e_{1} (\psi_{2}) - e_{3} (\psi_{4}) + 3 \psi_{3}) \eta_{3} 
+ (e_{1} (\psi_{1}) + e_{2} (\psi_{2}) + e_{3} (\psi_{3}) + 3 \psi_{4}) \eta_{4}, 
\end{align*}
which gives the proof. 
\end{proof}

From these descriptions, we compute the following.

\begin{prop} \label{sol_A3}
$\dim_{\mathbb{R}} \{ \psi \in C^{\infty}({\rm SU}(2), \mathbb{R}^{4}) ; D\psi = - \psi \} = 34$. 
\end{prop}

On the other hand, 
${\rm Spin(7)}$ induces 
$18 (= \dim_{\mathbb{R}} {\rm Spin}(7)/ {\rm SU} (2))$-dimensional 
associative deformations of $A_{3}$, 
and hence $A_{3}$ could potentially have $16$-dimensional nontrivial associative deformations. 
However, we do not know whether there exists actual 16-dimensional 
nontrivial deformations.

\begin{proof}[Proof of Proposition \ref{sol_A3}]
By Proposition \ref{explicit_D_A3}, $D \psi = \alpha \psi$ 
for $\alpha \in \mathbb{R}$ is equivalent to 
\begin{align}
\left( i e_{3} - \left ( \frac{15}{7} + \alpha \right) \right) \Psi_{1} + (-i e_{1} + e_{2}) \Psi_{2} = 0, \label{D_A3_1}\\
- (i e_{1} + e_{2}) \Psi_{1}                     +(-i e_{3} + (3 - \alpha))\Psi_{2} = 0. \label{D_A3_2} 
\end{align}
Applying $(i e_{1} + e_{2})$ to (\ref{D_A3_1}), we obtain 
\begin{align}
\left( i e_{3} - \left( \frac{1}{7} + \alpha \right) \right) ( i e_{1} + e_{2}) \Psi_{1} + \left( e_{1}^{2} + e_{2}^{2} + 
\frac{2}{7} i e_{3} \right) \Psi_{2} = 0.  \label{D_A3_3}
\end{align}
Substituting (\ref{D_A3_2}) into (\ref{D_A3_3}), we have 
$
\left( -7 \Delta_{+} + 24 i e_{3} + (7 \alpha + 1)(\alpha - 3) \right ) \Psi_{2} = 0. 
$
By using the notation in Lemma \ref{Fourier_S3} and Lemma \ref{calc_S3}, 
we obtain 
\begin{align*}
&\left( -7 \Delta_{+} + 24 i e_{3} + (7 \alpha + 1)(\alpha - 3) \right ) 
\langle \rho_{n} (\cdot) v^{(n)}_{k}, u \rangle \\
=& 
\left \{
- 6 (n - 2k + 2)^{2} -n^{2} - 2n + 24 + (7 \alpha + 1)(\alpha - 3)
\right \} 
\langle \rho_{n} (\cdot) v^{(n)}_{k}, u \rangle, 
\end{align*}
for $n \geq 0, 0 \leq k \leq n, u \in V_{n}.$

Set $\alpha = -1$. 
Since we know that 
$
- 6 (n - 2k + 2)^{2} -n^{2} - 2n + 48 = 0 
\Leftrightarrow 
(n, k) = (6, 4), (4, 2), (4, 4), 
$
we deduce that 
\begin{align*}
\Psi_{2} = 
\langle \rho_{6} (\cdot) v^{(6)}_{4}, u_{1} \rangle 
+ \langle \rho_{4} (\cdot) v^{(4)}_{2}, u_{2} \rangle
+ \langle \rho_{4} (\cdot) v^{(4)}_{4}, u_{3} \rangle, 
\end{align*}
for $u_{1} \in V_{6}, u_{2}, u_{3} \in V_{4}$. 
From (\ref{D_A3_1}), we see that 
\begin{align*}
\Psi_{1} = 
-i \sqrt{\frac{7}{10}} \langle \rho_{6} (\cdot) v^{(6)}_{5}, u_{1} \rangle 
- 2i  
\sqrt{\frac{7}{6}}
\langle \rho_{4} (\cdot) v^{(4)}_{3}, u_{2} \rangle. 
\end{align*}
Hence we obtain 
$\dim_{\mathbb{R}} \{ \psi \in C^{\infty}({\rm SU}(2), \mathbb{R}^{4}) ; D\psi = - \psi \} 
= 14 + 2 \cdot 10 
= 34$. 
\end{proof}

%%%%%%%%%%%%%%%%%%%%%%%%%%%%%%%%%%%%%%%%%%%%%%%%%%%%%%%%%%%%%%%%%%%%%%%%%%%%%%%%%%%%

\subsection{The case $S^{3}, L_{1}, L_{2}, L_{3}$ and $L_{4}$}

Next, we study the deformations of 
homogeneous associative submanifolds 
which lie in a totally geodesic $S^{6}$. 
These Lagrangian deformation spaces are studied in \cite{Lotay_stab}. 
Hence we only consider associative and non-Lagrangian deformations 
by Remark \ref{dim asso_nonLag}.

\subsubsection{The totally geodesic $S^{3}\cong {\rm SU}(2)$} \label{totally_geodesic_S3} \label{S3_deform}

In this case,  $\{ e_{1}, e_{2}, e_{3} \} = \{ E_{1}, E_{2}, E_{3} \}$ gives an orthonormal basis 
of $\mathfrak{su}(2)$
with respect to the induced metric from the totally geodesic $S^{3}$. 
We easily see the following by (\ref{gen_rot5}).

\begin{prop} \label{sol_lap_S3}
$\dim_{\mathbb{R}} \{ f \in C^{\infty}(S^{3}) ; \Delta_{+} f = 3f \} = 4$. 
\end{prop}

This implies that associative and non-Lagrangian deformations of the totally geodesic $S^{3}$ are trivial 
since $G_{2}$ induces 
$8 (= \dim_{\mathbb{R}} G_{2} / {\rm SO}(4))$-dimensional Lagrangian deformations of $S^{3}$ and 
${\rm Spin(7)}$ 
induces 12-dimensional associative deformations of $S^{3}$,  
whose space is known to be 
${\rm Spin(7)} / K$, 
where $K \cong {\rm SU}(2)^{3} / \mathbb{Z}_{2}$ is 
a Lie subgroup of {\rm Spin(7)} 
(\cite[Theorem.I\hspace{-.1em}V.1.38]{Harvey Lawson}). 
%

%%%%%%%%%%%%%%%%%%%%%%%%%%%%%%%%%%%%%%%%%%%%%%%%%%%%%%%%%%%%%%%%%%%%%%%%%%%%%%%%%%

\subsubsection{The case $L_{1} \cong {\rm SU}(2)$}

Set $p_{0} = \frac{\sqrt{5}}{3} \epsilon_{1} + \frac{2}{3} \epsilon_{4} 
= 
{}^t\! (\frac{\sqrt{5}}{3}, 0, \frac{2}{3}, 0) \in \mathbb{R} \oplus \mathbb{C}^{3}$. 
Then we have 
\begin{align*}
&(E_{1}^{*})_{p_{0}} = -\frac{2\sqrt{5}}{3} \epsilon_{3} - \frac{2}{3} \epsilon_{6}, \qquad
(E_{2}^{*})_{p_{0}} = -\frac{2\sqrt{5}}{3} \epsilon_{2} - \frac{2}{3} \epsilon_{7}, \qquad
(E_{3}^{*})_{p_{0}} =  \frac{2}{3} \epsilon_{5}. 
\end{align*} 
Thus $\{ e_{1}, e_{2}, e_{3} \} = 
\{ \frac{\sqrt{6}}{4} E_{1}, \frac{\sqrt{6}}{4} E_{2}, \frac{3}{2} E_{3} \}$ 
gives an orthonormal basis 
of $\mathfrak{su}(2)$. 
We easily see the following by (\ref{gen_rot5}).

\begin{prop} \label{sol_lap_L1}
$\dim_{\mathbb{R}} \{ f \in C^{\infty}(S^{3}) ; \Delta_{+} f = 3f \} = 7$. 
\end{prop}

This implies that associative and non-Lagrangian deformations of $L_{1}$ are trivial 
since ${\rm Spin(7)} \setminus G_{2}$ 
induces 
$7$-dimensional associative deformations of $L_{1}$.  
%

%%%%%%%%%%%%%%%%%%%%%%%%%%%%%%%%%%%%%%%%%%%%%%%%%%%%%%%%%%%%%%%%%%%%%%%%%%%%%

\subsubsection{The case $L_{2} \cong {\rm SU}(2)/\mathbb{Z}_{2}$}

Set $p_{0} = \frac{1}{\sqrt{2}} (\epsilon_{4} + \epsilon_{7}) 
= \frac{1}{\sqrt{2}} {}^t\! (0, 0, 1, i) \in \mathbb{R} \oplus \mathbb{C}^{3}$. 
Then we have 
\begin{align*}
(E_{1}^{*})_{p_{0}} = \sqrt{2} \epsilon_{3}, \qquad 
(E_{2}^{*})_{p_{0}} = \sqrt{2} \epsilon_{2}, \qquad 
(E_{3}^{*})_{p_{0}} =  \sqrt{2} (\epsilon_{5} - \epsilon_{6}). 
\end{align*} 
Thus $\{ e_{1}, e_{2}, e_{3} \} = 
\{ \frac{1}{\sqrt{2}} E_{1}, \frac{1}{\sqrt{2}} E_{2}, \frac{1}{2} E_{3} \}$ 
gives an orthonormal basis 
of $\mathfrak{su}(2)$.

\begin{lem} \label{Hom_L2}
If $n$ is even, we have 
\begin{align*}
{\rm Hom}_{\mathbb{Z}_{2}} (V_{n}, \mathfrak{su}(2) \otimes_{\mathbb{R}} \mathbb{C}) 
=
{\rm Hom}_{\mathbb{C}} (V_{n}, \mathfrak{su}(2) \otimes_{\mathbb{R}} \mathbb{C}), \qquad
{\rm Hom}_{\mathbb{Z}_{2}} (V_{n}, \mathbb{C}) 
=
{\rm Hom}_{\mathbb{C}} (V_{n}, \mathbb{C}). 
\end{align*}
If $n$ is odd, 
both spaces are $\{ 0\}$. 
\end{lem}

From this Lemma, we see the following by (\ref{gen_rot5}).

\begin{prop} \label{sol_lap_L2}
$\dim_{\mathbb{R}} \{ f \in C^{\infty}(L_{2}) ; \Delta_{+} f = 3f \} = 6$. 
\end{prop}

This implies that associative and non-Lagrangian deformations of $L_{2}$ are trivial 
since ${\rm Spin(7)} \setminus G_{2}$ induces 
$6$-dimensional 
associative deformations of $L_{2}$. 
Note that 
$L_{2}$ is invariant under the action of 	
$\{ {\rm diag}(e^{-3it}, e^{it}, e^{it}, e^{it}); t \in \mathbb{R} \} \subset {\rm Spin}(7) \setminus G_{2}$.

%%%%%%%%%%%%%%%%%%%%%%%%%%%%%%%%%%%%%%%%%%%%%%%%%%%%%%%%%%%%%%%%%%%%%%%%%%%%%%%%%%%%%

\subsubsection{The case $L_{3} \cong {\rm SU}(2)/ A^{*}_{4}$}

We have 
\begin{align*}
&(E_{1}^{*})_{\epsilon_{2}} = \sqrt{10} \epsilon_{4} - \sqrt{6} \epsilon_{6}, \qquad
(E_{2}^{*})_{\epsilon_{2}} = \sqrt{10} \epsilon_{5} - \sqrt{6} \epsilon_{7}, \qquad
(E_{3}^{*})_{\epsilon_{2}} = -4 \epsilon_{3}. 
\end{align*} 
Thus $\{ e_{1}, e_{2}, e_{3} \} = 
\{ E_{1}/4, E_{2}/4, E_{3}/4 \}$ 
gives an orthonormal basis 
of $\mathfrak{su}(2)$. 

\begin{lem} \label{Z2A4_equiv}
\begin{align*}
{\rm Hom}_{A^{*}_{4}} (V_{6}, \mathbb{C}) 
= \mathbb{C} \left( (v^{(6)}_{1})^{*} - (v^{(6)}_{5})^{*} \right). 
\end{align*}
\end{lem}

\begin{proof}
Recall that $A^{*}_{4} $ is generated by $k_{1}, k_{2}, k_{3}$ 
in (\ref{generator A4}). 
Take
$
L = \sum_{l = 0}^{10} C_{l} (v_{l}^{(10)})^{*}  
\in 
{\rm Hom}_{\mathbb{C}} (V_{10}, \mathbb{C})$
where $C_{l i} \in \mathbb{C}$ 
and consider the condition 
\begin{align} \label{Ad_Z2A4}
L( \rho_{10}(k) v) = L(v), 
\end{align}
for $k \in A^{*}_{4}$ and $v \in V_{10}$.
As for $k = k_{1}, k_{2}$, (\ref{Ad_Z2A4}) is equivalent to 
\begin{align*}
(-1)^{l} i^{6} C_{l} = C_{l}, \qquad
(-1)^{l} C_{6-l} = C_{l}. 
\end{align*}
Thus $L$ is of the form 
$C \left( (v^{(6)}_{1})^{*} - (v^{(6)}_{5})^{*} \right)$ 
for $C \in \mathbb{C}$, 
and we see that 
$(v^{(6)}_{1})^{*} - (v^{(6)}_{5})^{*}$ is invariant by $k_{3}$. 
\end{proof}

\begin{prop} \label{sol_lap_L3}
$\dim_{\mathbb{R}} \{ f \in C^{\infty}(L_{3}) ; \Delta_{+} f = 3f \} = 7$. 
\end{prop}

This implies that associative and non-Lagrangian deformations of $L_{3}$ are trivial 
since ${\rm Spin}(7) \setminus G_{2}$ induces 
$7$-dimensional 
associative deformations of $L_{3}$. 

\begin{proof}
The solution $f$ of $\Delta_{+} f = 3f$ is contained in 
$
{\rm span}_{\mathbb{C}} 
\left \{ 
\langle 
\rho_{6} (\cdot) 
v^{(6)}_{a}, 
v^{(6)}_{b}  
\rangle; 
0 \leq a,b \leq 6
\right \} 
$
from (\ref{gen_rot5}). 
From Lemma \ref{Z2A4_equiv}, 
$A^{*}_{4}$ invariant solutions of $\Delta_{+} f = 3f$ are of the form 
$
f = \langle \rho_{6} (\cdot) (v^{(6)}_{1} - v^{(6)}_{5}), u 
\rangle
$ 
for 
$u \in V_{6}$. 
Imposing that $f$ is $\mathbb{R}$-valued, we have 
$\dim_{\mathbb{R}} \{ f \in C^{\infty}(L_{3}) ; \Delta_{+} f = 3f \} 
= 7. 
$
\end{proof}

%%%%%%%%%%%%%%%%%%%%%%%%%%%%%%%%%%%%%%%%%%%%%%%%%%%%%%%%%

\subsubsection{The case $L_{4} \cong {\rm SU}(2)/ D^{*}_{3}$}

We have 
\begin{align*}
(E_{1}^{*})_{\epsilon_{6}} = \sqrt{6} \epsilon_{2}, \qquad
(E_{2}^{*})_{\epsilon_{6}} = \sqrt{6} \epsilon_{3}, \qquad
(E_{3}^{*})_{\epsilon_{6}} = 6 \epsilon_{7}. 
\end{align*} 
Thus $\{ e_{1}, e_{2}, e_{3} \} = 
\{ E_{1}/\sqrt{6}, E_{2}/\sqrt{6}, E_{3}/6 \}$ 
gives an orthonormal basis 
of $\mathfrak{su}(2)$. 

\begin{lem} \label{D3_equiv}
The space 
${\rm Hom}_{D^{*}_{3}} (V_{n}, \mathbb{C}) $ 
is spanned by 
the following functions: 
\begin{enumerate}
\item In case $n = 6m$ where $m \in \mathbb{Z}_{\geq 0}$, 
\begin{align*}
(v^{(n)}_{3j})^{*} + (-1)^{j} (v^{(n)}_{n-3j})^{*} \qquad
\mbox{ for } \ 0 \leq j \leq m.
\end{align*}

\item In case $n = 6m+2$, 
\begin{align*}
(v^{(n)}_{3j+1})^{*} + (-1)^{j + 1} (v^{(n)}_{n-(3j+1)})^{*} \qquad
\mbox{ for } \ 0 \leq j \leq m.
\end{align*}

\item 
In case $n = 6m+4$, 
\begin{align*}
(v^{(n)}_{3j+2})^{*} + (-1)^{j} (v^{(n)}_{n-(3j+2)})^{*} \qquad
\mbox{ for } \ 0 \leq j \leq m.
\end{align*}
\end{enumerate}

In case $n \in 2 \mathbb{Z} + 1$, we have 
${\rm Hom}_{D^{*}_{3}} (V_{n}, \mathbb{C}) = \{ 0 \}.$ 
\end{lem}

\begin{proof}
Recall that $D^{*}_{3}$ is generated by $k_{4}, k_{5}$ in (\ref{generator D3}).
Take
$
L = \sum_{l = 0}^{n} C_{l} (v_{l}^{(n)})^{*} \otimes e_{i} 
\in 
{\rm Hom}_{\mathbb{C}} (V_{n}, \mathbb{C})$
where $C_{l i} \in \mathbb{C}$. 
Consider the condition 
(\ref{Ad_Z2A4}) for $k = k_{4}, k_{5}$, it is equivalent to 
\begin{align*}
(-1)^{n-l} C_{n-l} = C_{l},\qquad  
(e^{\frac{\pi i}{3}})^{n-2l} C_{l} = C_{l}. 
\end{align*}
Then we easily see Lemma \ref{D3_equiv}. 
\end{proof}

\begin{prop} \label{sol_lap_L4}
$\dim_{\mathbb{R}} \{ f \in C^{\infty}(L_{4}) ; \Delta_{+} f = 3f \} = 7$. 
\end{prop}

This implies that associative and non-Lagrangian deformations of $L_{4}$ are trivial 
since ${\rm Spin}(7) \setminus G_{2}$ induces 
$7$-dimensional 
associative deformations of $L_{4}$. 

\begin{proof}
From (\ref{gen_rot5}), 
the solution $f$ of $\Delta_{+} f = 3f$ is contained in 
the space spanned by 
$
\langle 
\rho_{6} (\cdot) 
v^{(6)}_{j}, 
v^{(6)}_{a}  
\rangle
$
where 
$
j=0,6 
$
and 
$0 \leq a \leq 6. 
$
From Lemma \ref{D3_equiv}, 
$D^{*}_{3}$ invariant solutions of $\Delta_{+} f = 3f$ are of the form 
$
f = \langle \rho_{6} (\cdot) (v^{(6)}_{0} + v^{(6)}_{6}), u 
\rangle 
$
for 
$
u \in V_{6}. 
$
Imposing that $f$ is $\mathbb{R}$-valued, we have 
$\dim_{\mathbb{R}} \{ f \in C^{\infty}(L_{3}) ; \Delta_{+} f = 3f \} 
= 7. 
$
\end{proof}

\section*{Appendices}
\appendix

\section{Proof of Proposition \ref{Dirac_Lap}}

We follow the proof of \cite{Gayet}. 
First, we show the following lemma. 

\begin{lem} \label{diff_cross}
For any vector fields $u, v, w, z, X \in \mathfrak{X}(Y)$, we have 
\begin{align*}
\nabla_{X} (u \times v) 
=&
(\nabla_{X} u) \times v + u \times (\nabla_{X} v) - \chi (X, u, v), \\
R(w, z) (u \times v) 
=&
(R(w, z) u) \times v + u \times (R(w, z) v) 
+ \varphi (z, u, v) w - \varphi (w, u, v)z \\
&- g(w, u) v \times z 
- g(w, v) z \times u 
+ g(z, u) v \times w 
+ g(z, v) w \times u.  
\end{align*}
When $M^{3} \subset Y$ is associative, we have 
$TM \times TM \subset TM$, $TM \times \nu \subset \nu$, and 
$\nu \times \nu \subset TM$. Thus for any 
$X, u, v \in C^{\infty} (M, TM), \eta \in C^{\infty} (M, \nu)$, we have 
\begin{align*}
\nabla_{X}^{\top} (u \times v) 
=&
(\nabla_{X}^{\top} u) \times v + u \times (\nabla_{X}^{\top} v) - (\chi (X, u, v))^{\top}, \\
\nabla_{X}^{\perp} (u \times \eta) 
=&
(\nabla_{X}^{\top} u) \times \eta + u \times (\nabla_{X}^{\perp} \eta) - (\chi (X, u, \eta))^{\perp}.
\end{align*}
\end{lem}

\begin{proof}
Let $\{ f_{k} \}_{k = 1, \cdots, 7}$ be any local orthonormal frame of $TY$. Then 
\begin{align*}
\nabla_{X} (u \times v) 
&= \sum_{i = 1}^{7} 
\{
(\nabla_{X} \varphi) (u, v, f_{i}) f_{i} 
+ \varphi  (\nabla_{X} u, v, f_{i}) f_{i}
+ \varphi  (u, \nabla_{X} v, f_{i}) f_{i}
\}\\
&=
- \chi (X, u, v) + 
(\nabla_{X} u) \times v + u \times (\nabla_{X} v)
\end{align*}
since $\nabla g = 0$ and $\nabla \varphi = * \varphi$. 
For $R(w, z) = \nabla_{w} \nabla_{z} - \nabla_{z} \nabla_{w} - \nabla_{[w, z]}$ , 
we see the following by a direct computation. 
\begin{align*}
\hspace{-0.5cm}
R(w, z) (u \times v) 
=
(R(w, z) u) \times v + u \times (R(w, z) v) 
- (\nabla_{w} \chi)(z, u, v) + (\nabla_{z} \chi) (w, u, v). 
\end{align*}
Then, the equation $\nabla_{w} \chi = \sum_{k} i(f_{k}) (\nabla_{w} * \varphi) \otimes f_{k}
= - \sum_{k}  i(f_{k}) (g(w, \cdot) \wedge \varphi) \otimes f_{k}
= - \varphi \otimes w + \sum_{k}  (g(w, \cdot) \wedge i(f_{k}) \varphi) \otimes f_{k}
$
proves the lemma. 
\end{proof}

Next, we compute $D^{2}$. 
Let $\{ e_{i} \}_{i = 1, \cdots, 3}$ be any local orthonormal frame  
satisfying $e_{3} = e_{1} \times e_{2}$ and 
$\{ \eta_{k} \}_{k = 1, \cdots, 4}$ 
be any local orthonormal frame of $\nu$. 
Then by Lemma \ref{diff_cross}, it follows that  
\begin{align*}
D^{2} \psi 
= \sum_{i, j = 1}^{3} e_{i} \times \nabla_{e_{i}}^{\perp} (e_{j} \times \nabla_{e_{j}}^{\perp} \psi) 
= I_{1} + I_{2}, 
\end{align*}
where 
\begin{align*}
I_{1} &= \sum_{i, j = 1}^{3} e_{i} \times  ( \nabla_{e_{i}}^{\top} e_{j} \times \nabla_{e_{j}}^{\perp} \psi
          +e_{j} \times \nabla_{e_{i}}^{\perp} \nabla_{e_{j}}^{\perp} \psi ), \\
I_{2} &= - \sum_{i, j = 1}^{3} e_{i} \times (\chi (e_{i}, e_{j}, \nabla_{e_{j}}^{\perp} \psi) )^{\perp}.
\end{align*}
From (\ref{chi_cross}), the following holds: 
\begin{align*}
I_{2} &= \sum_{i, j} e_{i} \times ((e_{i} \times e_{j}) \times \nabla_{e_{j}}^{\perp} \psi) \\
      &= -\sum_{i, j} (e_{i} \times (e_{i} \times e_{j})) \times \nabla_{e_{j}}^{\perp} \psi 
      = 2 \sum_{j} e_{j} \times \nabla_{e_{j}}^{\perp} \psi
      = 2D \psi. 
\end{align*}
By the computation in \cite{Gayet}, we have 
$
I_{1} = 
\nabla^{\perp *} \nabla^{\perp} \psi 
+\pi_{\mathcal{V}} (I_{3}) + I_{4}, 
$
where 
\begin{align*}
I_{3} = -\frac{1}{2} \sum_{i, j} (e_{i} \times e_{j}) \times R(e_{i}, e_{j}) \psi, \qquad
I_{4} = \sum_{i, j, k} g (A_{(e_{i} \times e_{j}) \times \eta_{k}} e_{i}, A_{\psi} e_{j}) \eta_{k}. 
\end{align*}

From the next lemma, we obtain Proposition \ref{Dirac_Lap}. 

\begin{lem}
\begin{align*}
I_{3} =  \sum_{i = 1}^{3} R(e_{i}, \psi) e_{i} + 3 \psi, \qquad
I_{4} = - \mathcal{A} \psi.
\end{align*}
\end{lem}

\begin{proof}
By using the relation 
$e_{i} \times e_{i + 1} = e_{i+2}$ for $i \in \mathbb{Z}/3$ 
and the Bianchi identity, 
we have 
\begin{align*}
I_{3} =& - \sum_{i \in \mathbb{Z}/3} e_{i} \times R(e_{i + 1}, e_{i +2})\psi\\
=& \sum_{i \in \mathbb{Z}/3} e_{i} \times 
         \left( R(\psi, e_{i + 1}) e_{i +2} + R(e_{i + 2}, \psi) e_{i + 1} \right), \\
e_{i + 2} \times R(e_{i + 1}, \psi) e_{i} 
=& R(e_{i + 1}, \psi) e_{i + 1} - (R(e_{i + 1}, \psi) e_{i + 2}) \times e_{i} \\
 &- \varphi (\psi, e_{i + 2}, e_{i}) e_{i + 1} + \varphi (e_{i + 1}, e_{i + 2}, e_{i}) \psi \\
=&  R(e_{i + 1}, \psi) e_{i + 1} + e_{i} \times (R(e_{i + 1}, \psi) e_{i + 2}) + \psi, 
\end{align*}
since 
$e_{i} \times e_{i + 1} = e_{i+2}$ for $i \in \mathbb{Z}/3$, 
$g(e_{i}, \psi) = 0$,  
and $\varphi (e_{i}, e_{i + 1}, e_{i + 2}) = 1$. Hence we obtain 
$
I_{3} =  \sum_{i = 1}^{3} R(e_{i}, \psi) e_{i} + 3 \psi. 
$
For $I_{4}$, we have by Lemma \ref{diff_cross}
\begin{align*}
A_{(e_{i} \times e_{j}) \times \eta_{k}} e_{i} 
=& - \nabla_{e_{i}}^{\top} ((e_{i} \times e_{j}) \times \eta_{k}) \\
=& - \nabla_{e_{i}}^{\perp} (e_{i} \times e_{j}) \times \eta_{k}
  -  (e_{i} \times e_{j}) \times (\nabla_{e_{i}}^{\top} \eta_{k})
  + \chi (e_{i}, e_{i} \times e_{j}, \eta_{k})^{\top} \\
=&
- \{ (\nabla_{e_{i}}^{\perp} e_{i}) \times e_{j} + e_{i} \times (\nabla_{e_{i}}^{\top} e_{j}) \} \times \eta_{k}
+ (e_{i} \times e_{j}) \times A_{\eta_{k}} e_{i} \\
&+ \chi (e_{i}, e_{i} \times e_{j}, \eta_{k})^{\top}. 
\end{align*}
Since an associative submanifold is minimal, it follows that 
$\sum_{i} \nabla_{e_{i}}^{\perp} e_{i} = 0$. 
Moreover, we see $\sum_{i} e_{i} \times \nabla_{e_{i}}^{\perp} e_{j} = 0$ for  $j = 1, 2, 3$
by the relation $e_{3} = e_{1} \times e_{2}$. 
Hence we obtain $I_{4} = I_{5} + I_{6}$, where
\begin{align*}
I_{5} = \sum_{i, j, k} g( (e_{i} \times e_{j}) \times A_{\eta_{k}} e_{i}, A_{\psi} e_{j}) \eta_{k},  \qquad
I_{6} = \sum_{i, j, k} g( \chi (e_{i}, e_{i} \times e_{j}, \eta_{k})^{\top}, A_{\psi} e_{j}) \eta_{k}. 
\end{align*}
It is shown that $I_{5} = - \mathcal{A} \psi$ in \cite{Gayet}. 
As for $I_{6}$, we compute 
$
\chi (e_{i}, e_{i} \times e_{j}, \eta_{k}) 
= \eta_{k} \times (e_{i} \times (e_{i} \times e_{j})) 
= \eta_{k} \times (- e_{j} + \delta_{i j} e_{i}), 
$
and obtain 
$\sum_{i} \eta_{k} \times (- e_{j} + \delta_{i j} e_{i}) = -2 \eta_{k} \times e_{j} \in C^{\infty}(M, \nu)$, 
which implies that $I_{6} = 0$. 
\end{proof}

\section{Harmonic analysis on a homogeneous vector bundle}

We give a summary of harmonic analysis on a homogeneous vector bundle
from \cite{Wallach}. 
\subsection{Homogeneous vector bundles}

\begin{definition} \label{def_homog vb}
Let $G$ be a Lie group and let $K$ be a closed subgroup of $G$. 
Set $M := G/K$. 
A vector bundle $E \rightarrow M$ is called a 
{\bf homogeneous vector bundle} 
if $G$ acts on $E$ on the left and the $G$-action satisfies: 
\begin{enumerate}
\item 
$g \cdot E_{x} = E_{g \cdot x}$ for $g \in G, x \in M$, 
\item 
$g \cdot : E_{x} \rightarrow E_{g \cdot x}$ is linear for $g \in G, x \in M$, 
\end{enumerate}
where $E_{x}$ is the fiber of $E$ at $x \in M$. 
\end{definition}

\begin{lem}
Let $(\tau, E_{0})$ be a finite dimensional representation of $K$. 
Then the associated vector bundle 
$E := G \times_{\tau} E_{0} = G \times E_{0} / \sim, $
where
$
(g, v) \sim (g \cdot k, \tau (k)^{-1} v)$,  is a homogeneous vector bundle over $M$. 
\end{lem}

All homogeneous vector bundles are described as above 
by the following lemma. 

\begin{lem}
Let $E \rightarrow M$ be a homogeneous vector bundle. 
Let $E_{0} = E_{eK}$ and 
$\tau : K \rightarrow {\rm End}(E_{0})$ be the 
induced action from 2 of Definition \ref{def_homog vb}. 
Then we have $E \cong G \times_{\tau} E_{0}$. 
\end{lem}

\subsection{Fourier series expansion}

Let $G$ be a compact Lie group, $K$ be a closed subgroup of $G$, 
$(\tau, E_{0})$ be a finite dimensional unitary representation of $K$, 
and $E \rightarrow M$ be the homogeneous vector bundle 
associated with $(\tau, E_{0})$. 
Assume that $M = G/K$ is orientable. 
Setting  
\begin{align*}
C(G, E_{0})^{(K, \tau)} := \{ f \in C(G, E_{0}) ; f(g \cdot k) = \tau(k)^{-1} f(g) 
\mbox{ for any }
g \in G, k \in K)\}, 
\end{align*}
we have the following. 

\begin{lem}
For $f \in C(M, E)$, define $\tilde{f} \in C(G, E_{0})^{(K, \tau)}$ 
by 
$\tilde{f}(g) = g^{-1} f(gK) \in E_{eK} \cong E_{0}$. 
Then the map $f \mapsto \tilde{f}$ 
gives an isomorphism 
$C(M, E) \cong C(G, E_{0})^{(K, \tau)}$. 
The map $f \mapsto \tilde{f}$ extends to 
the isomorphism $A : L^{2}(M, E) \xrightarrow{\cong} L^{2}(G, E_{0})^{(K, \tau)}.$
\end{lem}

Let $\hat{G}$ be the set of all equivalence classes of finite dimensional irreducible unitary 
representations 
of $G$. 
For each $\gamma = [(\pi_{\gamma}, V_{\gamma})] \in \hat{G}$, 
we assign a map 
$A_{\gamma} : V_{\gamma} \otimes {\rm Hom}_{K} (V_{\gamma}, E_{0}) \rightarrow C(G, E_{0})^{(K, \tau)}$, 
where 
${\rm Hom}_{K} (V_{\gamma}, E_{0}) = \{ L \in {\rm Hom} (V_{\gamma}, E_{0}) ; 
L(k \cdot v) = \tau(k) L(v) \mbox{ for any }  k \in K, v \in V_{\gamma}) \}$, 
by 
$A_{\gamma}(v \otimes L) (g) = L(g^{-1} \cdot v)$.

\begin{prop}[Fourier expansion] \label{Fourier_general}
The algebraic direct sum 
\begin{align*}
\sum_{\gamma \in \hat{G}} A_{\gamma}(V_{\gamma} \otimes  {\rm Hom}_{K} (V_{\gamma}, E_{0}))
\end{align*}
is uniformly dense in $C(G, E_{0})^{(K, \tau)}$ relative to the uniform topology. 
\end{prop}

\begin{lem}[Schur orthogonality relations] \label{Schur_orthog}
Let $(\pi, V)$ and $(\pi', V')$ be  irreducible unitary representations 
of a compact group $G$. 
Let $(\cdot, \cdot)$ and $(\cdot, \cdot)'$ 
be inner products on $V$ and $V'$, respectively. 
Then for $u, v \in V$ and $u', v' \in V'$, we have 
\begin{align*}
\int_{G} (\pi(g)u, v) \overline{(\pi'(g)u', v')'} dg = 
\left\{
\begin{array}{ll}
0                                         & (\pi  \not\simeq  \pi') \\
(u, u')\overline{(v, v')}/\dim V    & (\pi \simeq \pi'). 
\end{array}
\right.
\end{align*}
\end{lem}

\subsection{Homogeneous differential operators}

\begin{definition}
Let $G$ be a Lie group and let $K$ be a closed subgroup of $G$. 
Set $M = G/K$. 
Let $E \rightarrow M$ and $F \rightarrow M$ be homogeneous vector bundles, 
and $(\tau, E_{0})$ and $(\sigma, F_{0})$ 
be the representations of $K$ associated with $E$ and $F$, respectively. 

A differential operator $D : C^{\infty}(M, E) \rightarrow C^{\infty}(M, F)$ is called 
a {\bf homogeneous differential operator} 
if 
$g \cdot Df = D(g \cdot f)$ 
for $g \in G, f \in C^{\infty}(M, E)$. 
Here, $(g \cdot f)(x) = g f(g^{-1} x)$ 
for $x \in M, g \in G, f \in C^{\infty}(M, E)$ or $C^{\infty}(M, F)$. 
\end{definition}

\begin{rem} \label{1pt_homog_diff}
The map $D$ is completely determined 
by its value at a point, 
i.e., 
given $(Df)_{eK}$ for any $f \in C^{\infty}(M, E)$, 
we can determine 
$(Df)_{gK}$ for each $g \in G, f \in C^{\infty}(M, E)$. 
\end{rem}

We give an explicit description of the homogeneous differential operators. 

Let $U(\mathfrak{g}) = \otimes^{*} \mathfrak{g}/ I(\mathfrak{g})$, 
where $I(\mathfrak{g})$ is the two-sided ideal in $\otimes^{*} \mathfrak{g}$ 
generated by 
$\{ X \otimes Y - Y \otimes X - [X, Y] ; X, Y \in \mathfrak{g} \}$. 
(In other words, $U(\mathfrak{g})$ is the universal enveloping algebra of $\mathfrak{g}$.) 
Let $\xi : \otimes^{*} \mathfrak{g} \rightarrow U(\mathfrak{g})$ be the canonical projection 
and $U^{i} (\mathfrak{g}) := \xi (\sum_{k \leq i} \otimes^{k} \mathfrak{g})$.

Set $D(G)$ be the space of 
all left invariant differential operators on $G$. 
For any $X \in \mathfrak{g}$ and $f \in C^{\infty}(G)$, 
define $Xf \in C^{\infty}(G)$ by 
$Xf(g) = (d/dt) f(g \cdot \exp(tX)) |_{t=0}$. 
The map $X \mapsto (f \mapsto Xf)$ gives the 
inclusion  $\mathfrak{g} \hookrightarrow D(G)$, 
from which 
an isomorphism $U(\mathfrak{g}) \xrightarrow{\cong} D(G)$
is induced. 
 
\begin{lem}
The algebra $U(\mathfrak{g})$ is isomorphic to $D(G)$. 
If $\{ X_{1}, \cdots, X_{n} \}$ is a basis of $\mathfrak{g}$, 
then $\{ X_{1}^{m_{1}} \cdots X_{n}^{m_{n}} ; m_{j} \geq 0 \}$ 
forms a basis of $U(\mathfrak{g})$. 
\end{lem}

Similarly, 
for $L \otimes X \in {\rm Hom}(E_{0}, F_{0}) \otimes U(\mathfrak{g})$ 
and $f \in C^{\infty}(G, E_{0})$, 
set $(L \otimes X)f = L \cdot Xf$. 
Thus 
the element of 
${\rm Hom}(E_{0}, F_{0}) \otimes U(\mathfrak{g})$
is considered as a differential operator 
$C^{\infty}(G, E_{0}) \rightarrow C^{\infty}(G, F_{0})$. 

Let $K$ act on ${\rm Hom}(E_{0}, F_{0}) \otimes U(\mathfrak{g})$ 
as 
$\mu(k) (L \otimes X) = \sigma(k) L \tau(k)^{-1} \otimes {\rm Ad}(k) X$ 
for $L \in {\rm Hom}(E_{0}, F_{0})$ and $X \in U(\mathfrak{g})$. 
Then 
$(\mu, {\rm Hom}(E_{0}, F_{0}) \otimes U^{j}(\mathfrak{g}))$ is 
a representation of $K$ for each $j$. Setting 
\begin{align*}
({\rm Hom}(E_{0}, F_{0}) \otimes U(\mathfrak{g}))^{K}
=
\{ D \in {\rm Hom}(E_{0}, F_{0}) \otimes U(\mathfrak{g}); 
\mu(k) D = D \mbox{ for any } k \in K
\}, 
\end{align*}
we have the following. 

\begin{lem}
For any $D \in ({\rm Hom}(E_{0}, F_{0}) \otimes U(\mathfrak{g}))^{K}$, we have 
$D C^{\infty}(G, E_{0})^{(K, \tau)} \subset C^{\infty}(G, F_{0})^{(K, \sigma)}$. 
Conversely, 
if $D \in {\rm Hom}(E_{0}, F_{0}) \otimes U(\mathfrak{g})$ satisfies
$D C^{\infty}(G, E_{0})^{(K, \tau)} \subset C^{\infty}(G, F_{0})^{(K, \sigma)}$, 
then $\mu(k)D|_{C^{\infty}(G, E_{0})^{(K, \tau)}} = D|_{C^{\infty}(G, E_{0})^{(K, \tau)}}$. 
\end{lem}

\begin{definition}
Let $\mathfrak{k}$ be the Lie algebra of $K$. 
A homogeneous space $G/K$ is called {\bf reductive} 
if there exists an ${\rm Ad}(K)$-invariant vector subspace $\mathfrak{p} \subset \mathfrak{g}$
satisfying $\mathfrak{g} = \mathfrak{k} \oplus \mathfrak{p}$.  
\end{definition}

\begin{prop}
Suppose that $M = G/K$ is a reductive homogeneous space. 
Let $D : C^{\infty}(M, E) \rightarrow C^{\infty}(M, F)$ be a homogeneous differential operator 
of order $j$ and 
let $\tilde{D}$ be the corresponding map from $C^{\infty}(G, E_{0})^{(K, \tau)}$ 
to $C^{\infty}(G, F_{0})^{(K, \sigma)}$. 

Then there exists $\overline{D} \in {\rm Hom}(E_{0}, F_{0}) \otimes U(\mathfrak{g})$ 
so that 
$\overline{D}|_{C^{\infty}(G, E_{0})^{(K, \tau)}} = \tilde{D}$. 
If $K$ is compact, 
$\overline{D}$ may be taken to be in $({\rm Hom}(E_{0}, F_{0}) \otimes U(\mathfrak{g}))^{K}$. 
\end{prop}

\address{Department of Mathematics, Gakushuin University, 1-5-1, Mejiro, Toshima, 
Tokyo, 171-8588, Japan}
{kkawai@math.gakushuin.ac.jp}

\end{document}